\documentclass[11pt,reqno]{amsart}
\usepackage{amsmath,amssymb,amsthm}
\usepackage{hyperref}
\hypersetup{bookmarksdepth=3}
\usepackage[usenames,dvipsnames]{xcolor}
\usepackage[a4paper]{geometry}
\usepackage{enumitem}
\usepackage{graphicx}
\usepackage{colonequals}
\usepackage{caption}
\usepackage{subcaption}

\newcommand{\N}{\mathbb{N}}
\newcommand{\R}{\mathbb{R}}
\newcommand{\C}{\mathbb{C}}
\newcommand{\eps}{\varepsilon}
\newcommand{\bddeven}{{\mathrm b, \mathrm e}}

\renewcommand{\Re}{\operatorname{Re}}
\renewcommand{\Im}{\operatorname{Im}}

\DeclareMathOperator{\sech}{\mathrm{sech}}

\numberwithin{equation}{section}

\newtheorem{thm}{Theorem}[section]
\newtheorem{lem}[thm]{Lemma}
\newtheorem{cor}[thm]{Corollary}
\newtheorem{prop}[thm]{Proposition}
\theoremstyle{remark}
\newtheorem{rk}[thm]{Remark}

\title[Solitary solutions to Euler with piecewise constant vorticity]{Solitary solutions to the steady Euler equations with piecewise constant vorticity in a channel}
\author[]{Karsten Matthies \and Jonathan Sewell \and Miles H. Wheeler}

\begin{document}

\begin{abstract}
We consider a two-dimensional, two-layer, incompressible, steady flow, with vorticity which is constant in each layer, in an infinite channel with rigid walls.
The velocity is continuous across the interface, there is no surface tension or difference in density between the two layers, and the flow is inviscid. 
Unlike in previous studies, we consider solutions which are localised perturbations rather than periodic or quasi-periodic perturbations of a background shear flow. 
We rigorously construct a curve of exact solutions and give the leading order terms in an asymptotic expansion.
We also give a thorough qualitative description of the fluid particle paths, which can include stagnation points, critical layers, and streamlines which meet the boundary.
\end{abstract}
\maketitle
\tableofcontents
\newpage

\section{Introduction}

The incompressible Euler equations have been studied for over 250 years, but even when considering time-independent, two-dimensional flows, there is still much being discovered. See \cite{bedrossian:book}, or \cite{saffman:book} for a tighter focus on the steady case.
The time-independent case is of interest both for its own sake, and to better understand possible end states of the time-dependent problem.
For instance, steady states give counterexamples to the phenomenon of \emph{invicid damping}, where solutions of the time-dependent Euler equations decay in some sense, despite energy being conserved \cite{Bedrossian:invDamp,LinZeng:NoDampNearCouette}. 

The study of solutions to the Euler equations with piecewise constant vorticity is classical. There are broadly speaking two cases:  \emph{vortex patches}, where the vorticity has compact support, and \emph{vorticity fronts}, with unbounded regions of vorticity. 
Studies of these phenomena date back to the 19th century, for instance by Rayleigh~\cite{Rayleigh:instability}, Kelvin~\cite{thomson:columnar}, and Kirchhoff~\cite{kirchhoff1897} with a more recent wave of interest in vortex patches including work such as \cite{Burbea:vortexPatches, HmidiMateuVerdera:vortexPatches, Gomez-serrano:vortexpatches,bhm:quasiperiodic, ej:patches}.
Time-dependent results have also been shown for vorticity fronts \cite{HUNTER:burgers3,hunter:burgers4,Hunter:burgers2,hmvsz:Burgers}. 

A common setting for study of the Euler equations is in a two-dimensional channel. It is shown in \cite{HamelNadirashvili:NoStagImpliesShear} that in a channel, any solution to the steady Euler equations is either a shear flow, or has a stagnation point. 
This contrasts with water waves with a free upper boundary, which exhibit many interesting solutions without stagnation~\cite{WaterWavesSurvey}. Indeed, many of the standard techniques in the classical study of water waves assume the stream function is strictly monotone in the vertical variable, so cannot allow for stagnation points. However, there has been great progress in understanding non-monotone stream functions in recent years \cite{Wahlen:critLayer,MatiocAncaVoichita:criticalLayer,walshBuhlerShatah:wind,Wheeler:Coriolis,vavaruca:gravCapWaves}, often using the weaker assumption that the vorticity can be expressed as a \emph{vorticity function} of the stream function.

We consider a two-dimensional flow in a channel with piecewise constant vorticity, and find explicit expansions of exact solutions, while thoroughly describing their qualitative behaviour.
This behaviour includes stagnation points, which complements \cite{HamelNadirashvili:NoStagImpliesShear}, and since the boundary at which the vorticity changes is free, this problem also fits in with the literature on water waves with stagnation.
However, for some values of our parameters, there are solutions for which even the weaker assumption that a vorticity function exists does not hold.
We focus on localised solutions, which have many challenges not present in the periodic case.
Reformulating our problem as an evolution equation, with the horizontal spatial direction playing the role of time, allows us to use powerful tools from spatial dynamics. These are particularly useful when working outside a periodic regime, especially a centre manifold theorem of Mielke \cite{Mielke:CMT}, which, for sufficiently small localised solutions, allows us to reduce our PDE to an ODE.

\subsection{The problem}

We consider a two-dimensional flow which is incompressible and inviscid. The flow has two layers, each with the same constant density, but with different constant vorticities.
There is no surface tension between the two layers. 
The fluid domain is infinite in width, but bounded above and below by rigid walls of constant height. 
It is also steady, that is, the wave profile moves with constant speed. In other words, if $X$ is the horizontal coordinate, and $t$ is time, a steady solution depends on $(X,t)$ only through the combination $X-ct$, for some constant wave speed $c$.
Then moving to a co-moving frame of reference gives a time-independent flow.
We use dimensionless units for distance and time, so that the channel has height 1, and the difference in vorticity between the two layers is 1.

More precisely, given real constants $\omega_0, \omega_1$ with $\omega_0 - \omega_1=1$, and $h \in (0,1)$, we seek $\eta \in C^2(\R)$ which in turn defines regions $\Omega_0=\{(X,Y) \mid 0 < Y < h + \eta(X) \}$
and $\Omega_1=\{(X,Y) \mid h + \eta(X) < Y < 1 \}$. We also seek a velocity field $U,V \in C^1(\overline{\Omega_0}) \cap C^1(\overline{\Omega_1}) \cap C^0(\overline{\Omega_0 \cup \Omega_1})$. These must satisfy
\begin{subequations}\label{eqn:stream}
  \begin{alignat}{2}
	\label{eqn:stream:divfree}
  	\partial_X U + \partial_Y V &= 0 &\qquad& \text{ in } \Omega_{0} \cup \Omega_1 \\  	
  	\label{eqn:stream:lap}
    \partial_Y U - \partial_X V &= \omega_{i} &\qquad& \text{ in } \Omega_{i} \text{ for } i=0,1 ,
  \end{alignat}
where \eqref{eqn:stream:divfree} is the incompressibility condition, and \eqref{eqn:stream:lap} enforces the piecewise constant vorticity. The vorticity of each fluid element is constant, so fluid cannot enter or leave either of $\Omega_0$ and $\Omega_1$. This gives the kinematic boundary conditions
  \begin{alignat}{2}    
    \label{eqn:stream:kintop}
    V &= 0 &\qquad& \text{ on } Y=1 \\
    \label{eqn:stream:kinbot}
    V &= 0 &\qquad& \text{ on } Y=0 \\
    \label{eqn:stream:kinint}
    \eta_X U -V &= 0 &\qquad& \text{ on } Y=h+\eta(X). 
  \end{alignat}
\end{subequations}

\begin{figure}
     \centering
     \begin{subfigure}[b]{0.55\textwidth}
         \centering
         \includegraphics[width=\textwidth]{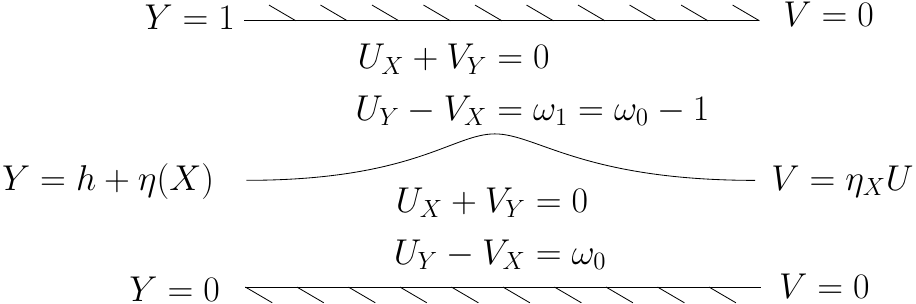}
         \caption{}
         \label{fig:setupandshear:setup}
     \end{subfigure}
     \hfill
     \begin{subfigure}[b]{0.35\textwidth}
         \centering
         \includegraphics[width=\textwidth]{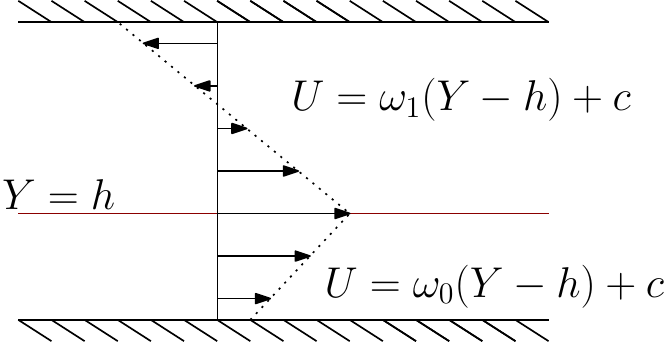}
         \caption{}
         \label{fig:setupandshear:shear}
     \end{subfigure}
        \caption{(A) The setup of the problem; see \eqref{eqn:stream} and \eqref{eqn:introShear}. (B) A shear solution in \eqref{eqn:shear}. }
        \label{fig:setupAndShear}
\end{figure}

We call $\{ (X,Y) \mid Y = h + \eta(X) \}$ the \emph{interface}. Points at which $U = V = 0$ are often of interest, and we refer to these as \emph{stagnation points}, or points at which the flow \emph{stagnates}. 
We use \emph{critical layers} to mean  curves along which the horizontal velocity $U$ vanishes  \cite{Wahlen:critLayer}. 
A solution is \emph{shear}, or a \emph{shear flow} if it has no $X$ dependence.
We say a solution is \emph{localised}, or \emph{homoclinic}, or forms a \emph{solitary wave} if $\lim_{X \to \infty} \eta(X)=\lim_{X \to -\infty} \eta(X) = 0$. A solitary wave with $\eta(X)>0$ for all $X$ is called a \emph{wave of elevation}, and one with $\eta(X)<0$ for all $X$, a \emph{wave of depression}. 

\subsection{Main results}

We construct localised solutions to \eqref{eqn:stream} by bifurcating from a shear flow. Localised solutions present many challenges that periodic solutions do not. 
For instance, since periodic solutions can be considered to have a compact domain, one has access to compact embeddings between H\"older and Sobolev spaces.
Furthermore, small-amplitude periodic solutions are a linear phenomenon, whereas small-amplitude solitary waves are weakly-nonlinear; see \eqref{eqn:ordersOfTildeF} and the surrounding discussion. 
 
We now, slightly informally, state our main result. Let the \emph{trivial solution}, which we will bifurcate from, be given by
\begin{equation}\label{eqn:introShear}
\begin{aligned}
    U(X,Y) &= 
    \begin{cases}
      \omega_0(Y-h)+ \tilde c \quad \text{for} \quad  0 \leq Y \leq h\\
        \omega_1(Y-h)+ \tilde c \quad \text{for} \quad  h \leq Y \leq 1\\
    \end{cases}
    \\
    V(X,Y)&=0\\
    \eta(X)&=0,\\
\end{aligned}
\end{equation}
where the constant $ \tilde c$ is the speed of the flow at the interface; see Figure~\ref{fig:setupAndShear}. 
\begin{thm}\label{thm:main}
Fix $N \in \N$ with $N \geq 4$, $h \in (0,1)$, and $\omega_0, \omega_1 \in \R$ such that $\omega_0 - \omega_1=1$ and
\begin{align}\label{eqn:defineTheta}
    \theta \colonequals  (3h-1)\omega_0-(3h-2)\omega_1 \neq 0.
\end{align}
Let $U^*$ be the horizontal velocity component of the shear flow \eqref{eqn:introShear} with $\tilde c=h(1-h)$.
Then, for $\eps>0$ sufficiently small, there exists a localised solution $(U^\eps, V^\eps, \eta^\eps)$ of \eqref{eqn:stream} that is close to $(U^*,0,0)$ and is approximated by the functions
\begin{equation}\label{thm:main:asym}
\begin{aligned}
    \overline{U}^\eps (X,Y) &= 
    \left\{
        \begin{alignedat}{3}
        &U^*(Y) +\eps  -(1-h)\overline{\eta}^\eps(X) & \quad \textup{for} \quad  Y \in [0, h+\overline{\eta}^\eps(X)]\\
        &U^*(Y) +\eps  +h\overline{\eta}^\eps(X) &  \quad \textup{for} \quad  Y \in [ h + \overline{\eta}^\eps(X),1] \\
      \end{alignedat}
      \right. \\
    \overline{V}^\eps(X,Y)&=\int_0^Y \overline{U}^\eps_X(X,\tilde Y) \; d \tilde Y\\
    \overline{\eta}^\eps(X) &= -\frac{3\eps}{\theta} \sech^2 \Big(\frac{\sqrt{3\eps}}{2h(1-h)} X \Big),
\end{aligned}
\end{equation}
in the sense that
\begin{equation}\label{thm:main:bounds}
    \lVert U^\eps-\overline{U}^\eps\rVert_{L^\infty} = \mathcal{O}(\eps^2),        \quad \lVert V^\eps-\overline{V}^\eps\rVert_{L^\infty} = \mathcal{O}(\eps^{\frac 52}),    \quad \lVert \eta^\eps-\overline{\eta}^\eps\rVert_{L^\infty} = \mathcal{O}(\eps^2).
\end{equation}
The velocity components $U^\eps,V^\eps$ are real-analytic away from the interface, and $\eta^\eps \in C^N(\R)$. Furthermore, $U^\eps$ and $\eta^\eps$ are even in $X$ while $V^\eps$ is odd in $X$.
\end{thm}
In the terminology of Lin and Zeng \cite{LinZeng:NoDampNearCouette}, there are steady structures which are an arbitrarily small, but non-zero, localised perturbation from a trivial shear flow.
\begin{rk}
    Each of the solutions described in Theorem~\ref{thm:main} in fact gives rise to an infinite family of solutions through phase shifts, i.e., through transformations of the form $X \mapsto X + a$.
\end{rk}
\begin{rk}\label{rk:vortfun}
    Let $\Psi^\eps$ be the stream function (see Lemma~\ref{lem:psi}) associated with $(U^\eps,V^\eps)$. If $\omega_0<1-2h$ or $\omega_0>2-2h$, then for $\eps$ sufficiently small there does not exist a vorticity function $\gamma^\eps$ such that $\omega=\gamma^\eps(\Psi^\eps)$. If $1-2h<\omega_0<2-2h$, on the other hand, then there does exist a vorticity function. See Proposition~\ref{prop:nonStreamFunctionVorticity} for more details.
\end{rk}

We can also give a relatively complete picture of what the flow looks like, and where and how it stagnates. All the solutions constructed in Theorem~\ref{thm:main} have a unique interior stagnation point, and some have two further stagnation points on the boundary.  Much of the qualitative behaviour can be determined by the sign of $\theta$, and whether $\omega_0$ is larger or smaller than the critical value $1-h$.

\begin{thm}\label{thm:main2}
    Every non-shear solution constructed in Theorem~\ref{thm:main} has either a unique stagnation point in $\Omega_0 \cup \Omega_1$, or one stagnation point in $\Omega_0 \cup \Omega_1$ and two saddle points on the boundary closest to the interior stagnation point. 
    The uniqueness, nature and location of the interior stagnation point, and whether $\eta^\eps$ gives a wave of elevation or depression, depends only on $\theta$ and $\omega_0$, as laid out in Table~\ref{tab:solutions in regions of parameter space}. Moreover $\eta^\eps$ is strictly monotone on $X>0$ and on $X<0$. 
    The solution has a critical layer, which is unbounded if $\omega_0 \neq 1-h$, and bounded if $\omega_0=1-h$. If $\omega_0=1-h$, there is a streamline which connects the saddle points on the boundary.
    See Figure~\ref{fig:solutions in regions of parameter space}.
\end{thm}
\begin{table}[htp]
    \centering
\begin{tabular}{r|l|l|l|l}
Region & Sign of $\theta$ & Size of $\omega_0$ & Wave profile & \begin{tabular}[c]{@{}l@{}}Location and nature \\ of the stagnation point\end{tabular} \\ \hline 
(i) & $\theta<0$ & $\omega_0 < 1-h$ & Elevation & Upper layer, unique saddle \\
(ii) & $\theta<0$ & $\omega_0 = 1-h$ & Elevation & Lower layer, non-unique centre \\
(iii) & $\theta<0$ & $\omega_0 > 1-h$ & Elevation & Lower layer, unique centre \\
(iv) & $\theta>0$ & $\omega_0 < 1-h$ & Depression & Upper layer, unique centre \\
(v) & $\theta>0$ & $\omega_0 = 1-h$ & Depression & Upper layer, non-unique centre \\
(vi) & $\theta>0$ & $\omega_0 > 1-h$ & Depression & Lower layer, unique saddle \\ 
\end{tabular}
    \caption{Behaviour of solutions in different regions of parameter space; see Figure~\ref{fig:solutions in regions of parameter space}. The stagnation point being referred to is the interior stagnation point.}
    \label{tab:solutions in regions of parameter space}
\end{table}
\begin{figure}[htp]
\includegraphics[scale=1]{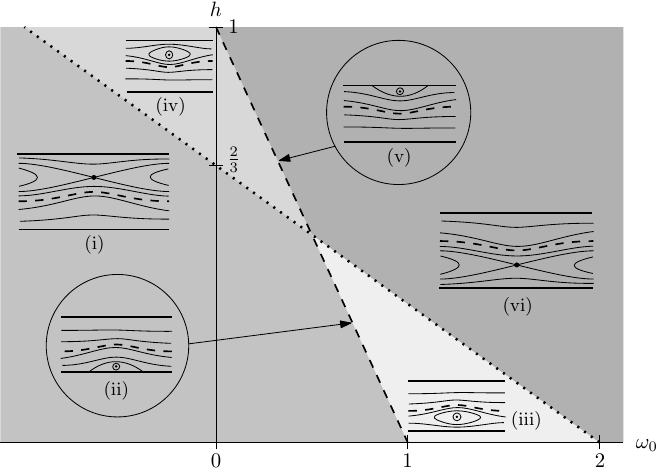}
\caption{\label{fig:solutions in regions of parameter space} Streamlines of solutions in the six different regions of parameter space from Table~\ref{tab:solutions in regions of parameter space}. The dashed streamlines correspond to the interface $Y=h+\eta(X)$. The dotted line in parameter space corresponds to $\theta=0$, we have not constructed solutions here.}
\end{figure}

\subsection{Related work}

One body of work that this problem is related to is that of travelling water waves with critical layers, as we are considering the Euler equations with a free boundary. See \cite{WaterWavesSurvey} and the references within for a general overview of the steady water waves literature, including in particular, problems with solutions exhibiting critical layers. 
Wahl\'en \cite{Wahlen:critLayer} gave the first exact construction of a rotational flow with critical layers and a free surface when considering the one-layer, constant vorticity case.  Since then, new results have taken this line of inquiry in different directions. For example, Ehrnstr\"om, Escher, Villari, and Wahl\'en \cite{EhrnströmWahlen:multipleCritlayers,EhrnströmEscherVillari:multipleCritLayers} found solutions with an arbitrary number of critical layers, with stagnation points exhibiting ``cat's eye'' structures. 
Matioc then moved away from these constant or affine vorticity functions, and considered a two-layer problem, where the two layers had distinct, constant vorticities and densities \cite{MatiocAncaVoichita:criticalLayer}.
Walsh, B\"uhler, and Shatah considered gravity-vorticity waves, \cite{walshBuhlerShatah:wind}, where they took the lower layer to be of finite depth and zero vorticity, and the upper layer to be either of infinite height and constant vorticity, or finite height and a more general vorticity.  
The study of critical layers is not limited to local perturbations, for instance, Constantin, Strauss, and V\u{a}rv\u{a}ruc\u{a} \cite{vavaruca:gravCapWaves}  considered gravity waves of constant vorticity, and used global bifurcation to construct large amplitude, overhanging waves. 
Wheeler considered a flow with an arbitrary number of layers, with waves of vorticity~ \cite{Wheeler:Coriolis}.
Whereas \cite{Wahlen:critLayer, EhrnströmWahlen:multipleCritlayers, EhrnströmEscherVillari:multipleCritLayers, vavaruca:gravCapWaves} considered the case when a vorticity function exists, we can work outside this scope.
More generally, water wave problems usually have a discontinuity in velocity across free surfaces, due to surface tension, or a change in density \cite{MatiocAncaVoichita:criticalLayer, walshBuhlerShatah:wind, Wheeler:Coriolis}. We do not have such a ``jump'' condition.

As already mentioned, our work here is also connected to invicid damping, where small smooth perturbations from a shear flow decay to a shear flow when evolved under the Euler equations. Bedrossian and Masmoudi were the first to demonstrat this phenomenon at the non-linear level \cite{Bedrossian:invDamp}.
Recently \cite{sinambela2023transition}, Sinambela and Zhao constructed less regular flows for which linear invicid damping occurs.
For a general overview see the survey \cite{inviscidDampingSurveyPaper}.
Invicid damping is an analogue of Landau damping; see \cite{Mouhot:landauDamping}.

Results on invicid damping are often complemented by results on the \emph{flexibility} or \emph{rigidity} of a flow. If a given steady solution of the Euler equations, often a shear flow, has nearby steady solutions, this is said to be flexible, if not it is rigid. If one can show flexibility near a shear flow, that is, that steady non-shear flows exist arbitrarily close to a shear flow, then invicid damping cannot hold. We go on to show that such a family of solutions exists.
Lin and Zeng \cite{LinZeng:NoDampNearCouette} showed that any periodic travelling steady solution with vorticity sufficiently close to a constant must be a shear flow. They also showed a flexibility result, that for arbitrary horizontal period, there exists steady flows in any neighbourhood of the Couette flow in $H{^\frac 52}$.
Constantin, Drivas, and Ginsburg \cite{constantin:flexibility} showed that given modest regularity assumptions, shear, steady, non-stagnating flows in a channel can be perturbed to give non-shear, steady, non stagnating flows in a perturbed channel.
Hamel and Nadirashvili \cite{HamelNadirashvili:NoStagImpliesShear} went on to show that in fact, in a channel, any steady flow which does not stagnate is a shear flow, and in the half plane, any flow which does not stagnate and has bounded velocity is a shear flow.
Zelati, Elgindi, and Widmayer consider waves of vorticity, \cite{czew}, and construct steady analytic flows arbitrarily close to a shear flow.
In \cite{sinambela2023transition}, Sinambela and Zhao considered a large family of shear flows, and constructed steady flows arbitrarily close in $H^s$ for arbitrarily large $s$. 
The recent work by Franzoi, Masmoudi, and Montalto \cite{Franzoi:quasiperiodic} constructs spatially quasiperiodic steady flows near the Couette flow. 

Most of the work mentioned so far
considered the periodic case, apart from \cite{Franzoi:quasiperiodic} which considers the quasiperiodic case.
As is often done in water wave problems, in order to consider localised solutions instead, we turn to spatial dynamics. 
This asks us to reformulate our PDE as an evolution equation, with $X$ playing the role of time, 
and turns out to yield many tools which are especially good at finding non-periodic solutions. 
One such spatial dynamical tool is a centre manifold reduction used by Kirchg\"assner~\cite{KirchgassnerJD1982}. However this can only be applied to semi-linear problems. Ours is a quasilinear problem, and for that, we need a stronger centre manifold reduction of Mielke~\cite{Mielke:CMT}. This allows us to reduce to an ODE. Note, this is an exact reduction, not an approximation; the evolution in $X$ of all sufficiently small bounded solutions to our infinite dimensional problem is confined to a two-dimensional manifold. Although we only encounter two-dimensional centre manifolds in this paper, with more than two layers we would expect to have higher dimensional manifolds such as those in~\cite{BuffoniTolandGroves:plethora}. 
There are great many papers applying spatial dynamics to water wave problems; see the surveys~\cite{groves:steadysurvey, Groves:GravCapSDsurvey} and the references therein for those focusing on gravity-capillary waves.
In one of the more classical papers, Kirchg\"assner used spatial dynamics to construct capillary-gravity waves~\cite{Kirchgassner:bifurcation}.
Groves, Toland, and Buffoni~\cite{BuffoniTolandGroves:plethora} considered the free surface problem and found solitary solutions lying on a four-dimensional manifold, from this centre manifold reduction of Mielke.
Groves and Wahl\'en constructed Stokes waves with vorticity using spatial dynamics in~\cite{GrovesWahlen:stokesWaves}.
More recently, Wang \cite{Wang:solitaryVorticityWaves} also uses this reduction, when considering a problem similar to ours, except with free upper boundary and fluids of different densities.
Kozlov, Kuznetsov, and Lokharu~\cite{KozlovV:solStag} found solitary gravity waves with vorticity and critical layers. They too considered the rotational case, and like us, they went on to find solutions with streamlines attached to the boundary.
In addition to classical spatial dynamics, recently-developed centre manifold techniques `without a phase space'~\cite{fs:centernophase,sc:tools} have been applied to problems related to water waves~\cite{tww:whitham,jtw:whitham,ChenWalshWheeler:WithoutPhaseSpace}; also see~\cite{bs:hamiltonian}. There are also many classical constructions of solitary water waves which employ fixed-point methods rather than spatial dynamics~\cite{lavrentiev:solitary, fh:solitary, beale:nashmoser, beale:ripples, sun:generalized}.

One body of work that also considers solutions to the two-dimensional Euler equations with continuous velocity but a discontinuity in vorticity is that of vortex patches. 
These are solutions of the Euler equations with compactly supported vorticity.
See the work by Burbea \cite{Burbea:vortexPatches}, by Hmidi, Mateu and Verdera \cite{HmidiMateuVerdera:vortexPatches}, by Hassainia and Wheeler \cite{wheeler:vortexPatches}, and by Castro, C\'ordoba and G\'omez-Serrano \cite{Gomez-serrano:vortexpatches}. Hassainia, Masmoudi and Wheeler \cite{HassainiaMasmoudiWheeler:catseyevortpatch} also find vortex patches with cat's eye structures, which also appear in many papers on water waves with critical layers.

A vorticity front can be thought of as a generalisation of vortex patches, where the vorticity has unbounded support.
Hunter, Moreno-Vasquez, Shu and Zhang consider a time dependent problem similar to ours, but without rigid walls \cite{hmvsz:Burgers}: they seek vorticity fronts, as do we, and they also consider perturbations from the same trivial flow, but extended to their infinite-depth domain.
It differs in that where we ultimately reduce to a perturbation of the Korteweg--De Vries equation, which is local in $x$, they reduce to equations involving a Hilbert transform, which is non-local in $x$. Furthermore, they consider the case where both layers are semi-infinite, and they consider the time-dependent case. They also find that their problem is non-dispersive. Subsequent work
\cite{Hunter:burgers2,HUNTER:burgers3,hunter:burgers4} considers an equation that approximates the evolution of the interface between vorticity fronts, and show that despite quadratic nonlinearity in the equation, the approximation holds for cubically non-linear time scales.

\subsection{Outline of the paper}
\label{sec:Outline}

Firstly, in Section~\ref{sec:prelims} we prove various preliminary results about the background shear flow, conserved quantities, and the stream function.
In Section~\ref{sec:reformulation} we change coordinates to reformulate the problem into a system of PDEs and boundary conditions on known domains. 
We then write our PDEs in the form of an evolution equation: $\frac \partial {\partial x} (u,v,\eta) = \mathcal{F}(u,v,\eta;c)$, where the only derivatives to appear in $\mathcal{F}$ are with respect to $y$. 
Finally, we write $\mathcal{F}$ as $L+ \mathcal R$, where $L$ is the Fréchet derivative of $\mathcal{F}$, and $ \mathcal R$ is the non-linear remainder term. 
In Section~\ref{sec:theCentreManifold} we prove Theorem~\ref{thm:main}. 
First, in Section~\ref{sec:L}, we verify some hypotheses about the spectrum of $L$. This allows us to apply the centre manifold theorem in Section~\ref{sec:CMT}, which lets us reduce the problem to solving an ODE rather than a PDE. 
Roughly speaking, the centre manifold theorem works by finding two things. 
The first is a \emph{reduction function}, $\psi\colon \mathcal{E}_0 \times \R \to \mathcal{W}$, where $\mathcal{W}$ is some space of functions in $y$, and $\mathcal{E}_0$ is the generalised kernel of $L$. 
Two functions, $\xi_0(y)$ and $\xi_1(y)$, span $\mathcal{E}_0$, so we often identify $\mathcal{E}_0$ with $\R^2$, and consider $\psi$ a function on $\R^3$.
The second is an ODE of the form $(a'(x),b'(x)) = F(a(x),b(x),\eps)$, where $\eps$ is a small parameter related to $c$.
These are such that if $(a,b,\eps)$ solves the ODE, then $(x,y) \mapsto a(x) \xi_0(y) + b(x) \xi_1(y) + \psi(a(x),b(x),\eps)(y)$ is a solution to our problem. 
The reduction function $\psi$ is not related to, and not to be confused with, the stream function $\Psi$.
In Section~\ref{sec:CMT} we find solutions to leading order, and prove Theorem~\ref{thm:main}.

In Section~\ref{sec:SignOfEta} we prove Theorem~\ref{thm:main2}. 
We show in Section~\ref{sec:signsofcomponents} that $\eta$ is monotone on $x>0$ and $x<0$, and that the sign of $V$ can be similarly characterised. In Sections~\ref{sec:Unbddstagnation} and~\ref{sec:bddstagnation} we show, given $h$, $\omega_0$, $\omega_1$, how and where the flow stagnates, and whether $\eta$ is deflected towards or away from this stagnation. In Section~\ref{sec:bddstagnation}, we deal with certain critical values, where we have that $U$ has the same sign in the entire flow, except for a bounded region of space which contains a stagnation point. 

In the appendix we have the details of an argument of Kirchg\"assner which we allude to in Section~\ref{sec:CMT}.

\section{Preliminaries}\label{sec:prelims}

We first introduce some notation and terminology.
Sometimes we write a function that depends on, say, $(X,Y,\omega)$. 
This unsubscripted $\omega$ should be taken to be equal to $\omega_0$ if $(X,Y) \in \Omega_0$, and equal to $\omega_1$ if $(X,Y) \in \Omega_1$. 
For an open set $\mathcal U$, then as usual $C^n( \mathcal U)$ denotes the set of $n$-times continuously differentiable functions with domain $\mathcal U$. 
By $C^n(\overline{\mathcal{U}})$, we mean the Banach space of functions which have domain $\mathcal U$, and whose derivatives up to order $n$ exist, can be continuously extended to $\overline{\mathcal{U}}$, and are bounded. To specify a codomain $\mathcal{V}$, we use the notations $C^n(\mathcal{U}, \mathcal{V})$ and $C^n(\overline{\mathcal{U}}, \mathcal{V})$ respectively.

As usual, incompressibility \eqref{eqn:stream:divfree} guarantees the existence of a \emph{stream function} $\Psi$ which is a first integral of the fluid particle motion.
\begin{lem}\label{lem:psi}
If $(U,V,\eta)$ solves \eqref{eqn:stream}, there exists $\Psi \in C^2(\overline{\Omega_0}) \cap C^2(\overline{\Omega_1}) \cap C^1(\overline{\Omega_0 \cup \Omega_1})$ such that $\Psi_X = -V$, $\Psi_Y = U$, and $\Psi=0$ on $Y=h+\eta(X)$.

\begin{proof}
Define $\Psi$ as the integral
\begin{align}
    \label{eqn:defnOfStreamfunction}
    \Psi(X,Y) = \int_{h+\eta(X)}^Y U(X, \tilde Y) \; d \tilde Y,
\end{align}
which clearly vanishes on $Y=h+\eta(X)$. Using the regularity of $U$ and $\eta$, it is straightforward to check that $\Psi \in C^0(\overline{\Omega_0 \cup \Omega_1})$ with $\Psi_Y = U$.  Differentiating \eqref{eqn:defnOfStreamfunction} with respect to $X$ and then integrating by parts using \eqref{eqn:stream:divfree}, we discover that
\begin{align*}
     \Psi_X(X,Y) &= \int_{h+\eta(X)}^Y U_X(X, \tilde Y) \; d \tilde Y
     - U(X,h+\eta(X))\eta_X(X)
     = -V(X,Y),
\end{align*}
where the boundary terms at $Y=h+\eta(X)$ cancel thanks to \eqref{eqn:stream:kinint}. The regularity $\Psi \in C^2(\overline{\Omega_0}) \cap C^2(\overline{\Omega_1}) \cap C^1(\overline{\Omega_0 \cup \Omega_1})$ follows immediately.
\end{proof}
\end{lem}

One important symmetry this problem has is that it is \emph{reversible} in $X$. More precisely, if $(U,V,\eta)$ solves \eqref{eqn:stream}, then the functions $(\check U, \check V, \check \eta)$ defined by
\begin{align}
    \label{eqn:UVreversibility}
    \check U(X,Y) =  U(-X,Y),\quad \check V(X,Y) = -V(-X,Y),\quad
    \check \eta(X) = \eta(-X),
\end{align}
also solve \eqref{eqn:stream}. Physically, this corresponds to reflecting the particle motion across the $Y$ axis and reversing time. Note that the solutions in Theorem~\ref{thm:main} are invariant under this symmetry.

We next introduce several well-known invariants for \eqref{eqn:stream}. Since $X$ plays a time-like role, these are called `conserved quantities' in what follows.
\begin{lem}[Conserved quantities]
\label{lem:PhysicalCQs}
Let $(U,V,\eta)$ be a solution to \eqref{eqn:stream}. Then the quantities 
\begin{align*}
Q_0 &= \int_0^{h+\eta(X)} U(X,Y) \; dY \\
Q_1 &=\int_{h+\eta(X)}^1 U(X,Y) \; dY \\
S &= \int_0^1 \left(  \frac{V(X,Y)^2 - U(X,Y)^2}{2}+\omega Y  U(X,Y) \right) \; dY
\end{align*} 
are constants independent of $X$.
\end{lem}
We call $Q_0$ the \emph{mass flux} of the lower layer and $Q_1$ the mass flux of the upper layer. The \emph{flow force} $S$ does not play a role in our later arguments, but it is included here for completeness.
\begin{proof}
First consider the mass flux of the lower layer, $Q_0$, and find that 
\begin{align*}
    \frac{d}{dX}\int_0^{h+\eta} U \; dY &= \eta_X U(X,h+\eta) + \int_0^{h+\eta} U_X \; dY\\
     &= \eta_X U(X,h+\eta) - \int_0^{h+\eta} V_Y \; dY\\
     &= \eta_X U(X,h+\eta) - V(X,h+\eta)
     = 0.
\end{align*}
The second equality is due to \eqref{eqn:stream:divfree}, and the last is due to \eqref{eqn:stream:kinint}. An almost identical proof shows the conservation of $Q_1$. 
Now consider the flow force. We see that
\begin{align*}
\frac{d S}{dX} &= \frac{d}{dX} \left( \int_0^{h+\eta} \left(  \frac{V^2 - U^2}{2}+\omega Y  U \right) \; dY +  \int_{h+\eta}^1 \left(  \frac{V^2 - U^2}{2}+\omega Y  U \right) \; dY \right) \\ 
&= \int_0^1 \big(  VV_X - UU_X+\omega Y  U_X \big) \; dY + (\omega_0-\omega_1)(h+\eta)U(X,h+\eta)\eta_X \\
&= \int_0^1 \big(  (U_Y-\omega)V + UV_Y-\omega Y V_Y \big) \; dY + (\omega_0-\omega_1)(h+\eta)U(X,h+\eta)\eta_X,
\end{align*}
where the last line uses \eqref{eqn:stream:divfree} and \eqref{eqn:stream:lap}. Performing integration by parts on the last term in the integrand, and using \eqref{eqn:stream:kinint} to eliminate the boundary terms, gives us 
\[
\frac{d S}{dX}= \int_0^1 (VU_Y + UV_Y) \; dY =\int_0^1  (UV)_Y \; dY =0.
\qedhere
\]
\end{proof}

Shear flows, that is, solutions of \eqref{eqn:stream} which do not depend on $X$, play a key role in the analysis. As we are viewing $X$ as a time-like variable, we call these solutions \emph{equilibria}. 

\begin{lem}[Equilibria]\label{lem:shear}
    All solutions to \eqref{eqn:stream} with no $X$ dependence are of the form
    \begin{equation}\label{eqn:shear}
    U=\omega(Y-h-\eta) + \tilde c, \quad V=0, \quad \eta \textup{ constant}.
\end{equation}
\begin{proof}
Suppose $(U,V,\eta)$ is an equilibrium solution to \eqref{eqn:stream}. By \eqref{eqn:stream:divfree}, \eqref{eqn:stream:kinbot}, and \eqref{eqn:stream:kintop}, we have that $V=0$. 
Furthermore, from the definition of equilibrium, $\eta$ is a constant. Finally, integrating \eqref{eqn:stream:lap} gives $U = \omega(Y-h-\eta) + \tilde c$, where $\tilde c $ is an arbitrary constant equal to the horizontal velocity at $y=h+\eta$.    
\end{proof}
\end{lem}
If, as in \eqref{eqn:introShear}, an equilibrium solution has $\eta = 0$, we call it a \emph{trivial} solution.
However, notice that for any general equilibrium of the form in \eqref{eqn:shear}, $h$ can be redefined to give an equilibrium of the form in \eqref{eqn:introShear}.

Many studies on steady solutions of the two-dimensional Euler equations assume the existence of a so-called `vorticity function' $\gamma$ such that the $\omega = U_Y - V_X = \gamma(\Psi)$. This then implies that the stream function $\Psi$ satisfies the semilinear equation $\Delta\Psi = \gamma(\Psi)$. Interestingly, as mentioned in Remark~\ref{rk:vortfun}, many of the solutions that we construct in Theorem~\ref{thm:main} do not possess a vorticity function.
\begin{prop}\label{prop:nonStreamFunctionVorticity}
Let $(U^*,0,0)$ be the shear solution given by $U^*=\omega(Y-h) + h(1-h)$, and let $(U,V,\eta)$ be a solution to \eqref{eqn:stream} with stream function $\Psi$ satisfying
\begin{equation}\label{eqn:vortfunbounds}
    \lVert U-U^* \rVert_{L^\infty} < \delta, 
    \qquad 
    (1+\lVert U^* \rVert_{L^\infty}) \lVert\eta \rVert_{L^\infty} <\delta
\end{equation}
for some sufficiently small $\delta$ depending only on $\omega_0$ and $h$. If $\omega_0<1-2h$ or $\omega_0>2-2h$, then there does not exist a single-valued function $\gamma$ such that $\omega=\gamma(\Psi)$. If $1-2h<\omega_0<2-2h$, then there does exist such a function.
\begin{rk}
    We focus on the shear flow $(U^*,0,0)$ as it corresponds to $\eps=0$ in Theorem~\ref{thm:main}.  
    In particular, using \eqref{thm:main:asym} and \eqref{thm:main:bounds} we see that Proposition~\ref{prop:nonStreamFunctionVorticity} applies to the solutions $(U^\varepsilon,V^\varepsilon,\eta^\varepsilon)$ in Theorem~\ref{thm:main}, provided $\omega_0 \ne 1-2h, 2-2h$ and $\varepsilon$ is sufficiently small. Similar arguments could be made for perturbations of more general shear flows.
\end{rk}
\begin{proof}[Proof of Proposition~\ref{prop:nonStreamFunctionVorticity}]
In what follows all stream functions are defined using \eqref{eqn:defnOfStreamfunction} so that they vanish at their respective interfaces. Thus the stream function corresponding to $(U^*,0,0)$ is 
\begin{equation*}
    \Psi^* = \Psi^*(Y) = \tfrac 1 2 \omega (Y-h)^2 + h(1-h)(Y-h),
\end{equation*}
and \eqref{eqn:vortfunbounds} implies the estimate
\begin{equation}\label{eqn:nonstreamest}
    \lVert \Psi^*-\Psi \rVert_{L^\infty} < \delta.
\end{equation}

First consider the case $\omega_0<1-2h$. Using \eqref{eqn:nonstreamest} and $\omega_1=\omega_0-1$, we find 
\begin{equation*}
    \Psi(0,1) < \Psi^*(1)+\delta 
    = \tfrac 12 (1-h)^2 (\omega_0 - 1 + 2h) + \delta < 0,
\end{equation*}
provided $\delta > 0$ is sufficiently small, and similarly
\begin{equation*}
    \Psi(0,0) < \Psi^*(0)+\delta
    = \tfrac 12 h^2 \omega_0 - (1-h)h^2 + \delta < - \tfrac 12 h^2 + \delta < 0.
\end{equation*}
Since $\Psi(0,h+\eta(0))=0$ by construction, the intermediate value theorem guarantees the existence of $Y_1 \in (h+\eta(0),1)$ and $Y_0 \in (0,h+\eta(0))$ with $\Psi(0,Y_0)=\Psi(0,Y_1)<0$. As $(0,Y_0) \in \Omega_0$ and $(0,Y_1) \in \Omega_1$, the corresponding vorticities $\omega(0,Y_0)=\omega_0$ and $\omega_1 = \omega(0,Y_1)$ are distinct, and hence there cannot exist a global vorticity $\gamma$ such that $\omega=\gamma(\Psi)$. The $\omega_0>2h-2$ follows by an analogous argument.

It remains to consider the case where $1-2h<\omega_0<2-2h$. We claim that $\Psi < 0$ in $\Omega_0$ and $\Psi > 0$ in $\Omega_1$. This will imply that $\omega = \gamma(\Psi)$ where $\gamma(t)=\omega_0 - H(t)$ and $H$ is the Heaviside step function.
For $\lvert Y - h \rvert \le \sqrt\delta$, these strict signs for $\Psi$ follow from the fact that $\Psi(X,\eta(X))=0$ together the uniform lower bound
\begin{align*}
    \Psi_Y(X,Y) = U(X,Y) > U^*(Y) - \delta \ge (1-h)h - \lVert \omega \rVert_{L^\infty}\sqrt \delta - \delta > 0,
\end{align*}
which holds for $\delta$ sufficiently small. For $\lvert Y - h \rvert > \sqrt\delta$, suppose first that $\omega_0 \ge 1 - h$. Then $\Psi^*$ has at most one critical point in $(0,h)$, which is a local minimum, and is strictly increasing on $(h,1)$. Thus, for $\delta>0$ sufficiently small, we have
\begin{align*}
    \max_{\R \times [0,h-\sqrt\delta]} \Psi 
    &<  \max_{[0,h-\sqrt\delta]}\Psi^* + \delta 
    < \max\{ \Psi^*(0), \Psi^*(h-\sqrt\delta)\} + \delta \\
    &\le \max\{ -\tfrac 12 h^2 + \delta, -\sqrt\delta h(1-h) + \delta(2-h)\} < 0
\end{align*}
and similarly
\begin{align*}
    \min_{\R \times [h+\sqrt\delta,1]} \Psi 
    &>  \min_{[h+\sqrt\delta,1]}\Psi^* - \delta 
    > \min\{ \Psi^*(h+\sqrt\delta), \Psi^*(1)\} - \delta \\
    &\ge \max\{ \sqrt\delta h(1-h) - \delta (1+h), \tfrac 12 (1-h)^2 - \delta \} > 0.
\end{align*}
The arguments for $\omega_0 \le 1-h$ are similar but with the role of the two layers reversed, and the claim is proved.
\end{proof}
\end{prop}

\section{Reformulation}\label{sec:reformulation}
\subsection{Pointwise flattening with a M\"obius map}\label{sec:reformulation:mobius}
One of main challenges of \eqref{eqn:stream} is that it is a free-boundary problem, i.e., the domains $\Omega_0$ and $\Omega_1$ are unknowns. 
To overcome this, we introduce new coordinates to map the problem onto a known domain, at the cost of making the PDE much more nonlinear. 

In particular, we seek a coordinate transformation which flattens interface to a straight line of constant height $h$, pointwise in $X$, and fixes the upper and lower boundaries. Thus the new independent variables $x = x(X,Y)$ and $y=y(X,Y)$ should satisfy
\[
x(X,Y)=X, \quad y(X,0)=0,\quad y(X,h+\eta(X))=h, \quad y(X,1) = 1.
\]
The obvious choice is to let $Y$ depend on $y$ in a piecewise linear fashion\footnote{Thus the new independent variables would be 
\[
	\tilde x = X , \quad 
	\tilde y = \begin{cases}
            \dfrac{hY}{h+\eta(X)} & \quad \text{for } 0 \leq Y \leq h+\eta(X)\\[3ex]
            \dfrac{(1-h)Y-\eta(X)}{1-h-\eta(X)} & \quad \text{for } h+ \eta(X) \leq Y \leq 1,
        \end{cases}\]
for instance with $\tilde u(x,y) = U(X,Y)$ and $\tilde v(x,y) = V(X,Y)$ as dependent variables.}. For our purposes, however, it is more convenient to have a globally smooth transformation, and so instead we use M\"obius functions to smoothly map $Y$ to $y$, defining
\begin{subequations}\label{eqn:coords}
  \begin{align}
	\label{eqn:coords:x} x &= X\\
	\label{eqn:coords:y} y &= \frac{h(1-(h+\eta(X)))Y}{(1-h)(h+\eta(X)) - \eta(X) Y}.
  \end{align}
  For independent variables we take
  \begin{align}
    \label{eqn:coords:u} u(x,y) &=  Y_y U(X,Y)  - \omega (y-h) - c \\
	\label{eqn:coords:v} v(x,y) &= V(X,Y),
  \end{align}
where $c$ is a parameter.
\end{subequations}
Notice that $u=v=\eta=0$ always gives a solution, in particular the trivial solution from \eqref{eqn:introShear} with $\tilde c =c$. 

We now differentiate \eqref{eqn:coords:y} to get
\begin{equation}\label{eqn:yderivatives}
\begin{aligned}
    y_X &= - \frac{\eta_{x} y (1-y)}{(h+\eta)(1-(h+\eta))}, &  y_Y &= \frac{(\eta y - \eta h + h(1-h))^2 }{h(1-h)(h+\eta)(1-(h+\eta))},  \\
	y_{XY} &= \eta_{x} \frac{(\eta y - \eta h + h(1-h))^2 (2y-1) }{h(1-h)(h+\eta)^2 (1-(h+\eta))^2}, &  y_{YY} &= \frac{2\eta (\eta y - \eta h + h(1-h))^3 }{h^2(1-h)^2(h+\eta)^2(1-(h+\eta))^2} \ ,
\end{aligned}    
\end{equation}
and rearrange and differentiate \eqref{eqn:coords:u} and \eqref{eqn:coords:v} to find
\begin{equation}\label{eqn:UVderivatives}
\begin{aligned}
  U_X &= y_{XY} (u + \omega (y-h) + c) + y_Y u_x + y_X y_Y (u_y   + \omega), & V_X &= v_x + y_X v_y,\\
  U_Y &= y_{YY} (u + \omega (y-h) + c) + y_Y^2 (u_y + \omega), & V_Y &= y_Y v_y.
\end{aligned}     
\end{equation}
\begin{subequations}\label{eqn:stream2}
Substituting \eqref{eqn:yderivatives} and \eqref{eqn:UVderivatives} into \eqref{eqn:stream:divfree}--\eqref{eqn:stream:lap} then gives
  \begin{align}
    \nonumber
  	0 &= \eta_{x}  (2y-1) (u + \omega(y-h) + c) + (h+\eta)(1-(h+\eta)) u_x 
    - \eta_{x} y (1-y) (u_y + \omega)
    \\
    \label{eqn:stream2:divfree2}
  	& \qquad + (h+\eta)(1-(h+\eta)) v_y,\\
     \nonumber
    \omega &= \frac{2 \eta (\eta y - \eta h + h(1-h))^3(u + \omega(y-h) + c)  
     + (\eta y - \eta h + h(1-h))^4 (u_y + \omega) }{h^2(1-h)^2(h+\eta)^2(1-(h+\eta))^2}  \\
  	\label{eqn:stream2:lap2}
     &\qquad - v_x + \frac{\eta_{x} y (1-y)}{(h+\eta)(1-(h+\eta))} v_y ,
  \end{align}
  for $y \in (0,h) \cup (h,1)$, while the boundary conditions \eqref{eqn:stream:kintop}--\eqref{eqn:stream:kinint} become
  \begin{alignat}{2}
    \label{eqn:stream2:kintop2}
    v &= 0 &\qquad& \text{ on } y=1, \\
    \label{eqn:stream2:kinbot2}
    v &= 0 &\qquad& \text{ on } y=0, \\
    \label{eqn:stream2:kinint2}
    \eta_x &= \frac{(h+\eta)(1-(h+\eta)) v}{h(1-h) (u+c)} &\qquad& \text{ on } y=h, 
  \end{alignat}
\end{subequations}
where $\eta \in C^2(\R)$, and $u,v \in C^1(\overline{\R \times (0,h)}) \cap C^1(\overline{\R \times (0,h)}) \cap C^0 \big(\overline{\R \times (0,1)} \big)$.

\subsection{Properties of this coordinate change}

One benefit of this change of coordinates is that there are conserved quantities related to the mass fluxes $Q_0,Q_1$ in Lemma~\ref{lem:PhysicalCQs} which have a particularly simple form. 
We let
\[ q_0 = \int_0^h u(x,y) \; dy, \qquad q_1 = \int_h^1 u(x,y) \; dy,\]
and call $q_0$ the \emph{pseudoflux} of the lower layer, and $q_1$ the pseudoflux of the upper layer. 
Notice that for solutions, these are conserved quantities, since
\begin{align*}
	\frac{d}{dx} q_0 &= \frac{d}{dx} \int_0^{h+\eta(X)} U(X,Y) Y_y  y_Y(X,Y) \; dY + \frac{d}{dx} \left( \frac{h^2 \omega_0}{2} - hc \right) \\
	&=\frac{d}{dX}\int_0^{h+\eta(X)} U(X,Y) \; dY
	= 0,
\end{align*}
and similarly for $q_1$, where in the last step we have used Lemma~\ref{lem:PhysicalCQs}.

We will ultimately restrict our attention to solutions with vanishing pseudofluxes, i.e.~to solutions satisfying the constraints
\begin{align}
    \label{eqn:lowermean}
    \int_0^h u \; dy &= 0\\
    \label{eqn:uppermean}
     \int_h^1 u \; dy &= 0.
\end{align}
The shear flows from \eqref{eqn:shear} with $\tilde c=c$ satisfy \eqref{eqn:lowermean}--\eqref{eqn:uppermean}. Furthermore, since the pseudofluxes are conserved quantities, the same is true for any flows which tend to one of these shear flows as $x \to -\infty$ or $x \to +\infty$, in particular the solitary waves constructed in Theorem~\ref{thm:main}. For a general solution $(U,V,\eta)$ in the original variables, without well-defined limits as $x \to \pm\infty$, we can always choose the value of $c$ in \eqref{eqn:coords:u} so that one of \eqref{eqn:lowermean} and \eqref{eqn:uppermean} is satisfied, but imposing both is an additional restriction.

While we do not need the flow force $S$ from Lemma~\ref{lem:PhysicalCQs} in our arguments, we nevertheless record its expression in terms of the new variables for completeness,
\[ S = \int_0^1 \bigg(\frac{1}{2}Y_y v^2 -\frac{(u - \omega(y-h) + c)^2}{2Y_y} + \frac{\omega y (h+\eta)(1-h)}{\eta(y-h)+h(1-h)}(u-\omega(y-h)+c)\bigg) \; dy. \]
Note while the pseudofluxes are related to -- but not exactly -- the mass fluxes, this is exactly the flow force, just written in the new coordinates.

The reversibility in \eqref{eqn:UVreversibility} is also preserved. That is, if $(u, v, \eta;c)$ satisfies \eqref{eqn:stream2}--\eqref{eqn:uppermean}, and we define
\begin{equation}\label{eqn:reversibility}
     \check u(x,y) = u(-x,y), \quad \check v(x,y) = -v(-x,y), \quad \check \eta(x) = \eta(-x),
\end{equation} 
then $(\check u,\check v,\check \eta;c)$ also satisfy \eqref{eqn:stream2}--\eqref{eqn:uppermean}.

Lastly, let us describe the $x$-independent solutions. In physical coordinates, with fixed $\omega$ and $h$, \eqref{eqn:stream} has a two-parameter family of equilibria given by \eqref{eqn:shear}. However, once we enforce the constraints \eqref{eqn:lowermean}--\eqref{eqn:uppermean} on the psuedofluxes, we are left with are most three equilibria.

\begin{lem}\label{thm:numberOfEquilibria}
For fixed $h,\omega,c$, the number of equilibrium solutions of \eqref{eqn:stream2}--\eqref{eqn:uppermean} is given by the number of real roots $\eta$ of
\begin{align}\label{eqn:cubicForEquilibriumEta}
    \eta(\eta^2 + \theta \eta + 2\eps) &=0,
\end{align}
where $\eps=c-h(1-h)$ and $\theta$ is defined in \eqref{eqn:defineTheta}.
\end{lem}
\begin{rk}
    There is always one equilibrium with $\eta=0$. Supposing as in Theorem~\ref{thm:main} that $\theta \ne 0$, there is one non-zero equilibrium when $\eps = 0$, while for $0<\eps<\frac 18 \theta^2$ there are two non-zero equilibria.
\end{rk}
\begin{proof}[Proof of Lemma~\ref{thm:numberOfEquilibria}]
By Lemma~\ref{lem:shear}, any equilibrium solution of \eqref{eqn:stream2} must correspond to an equilibrium solution \eqref{eqn:shear} of \eqref{eqn:stream}. Applying the coordinate transformation \eqref{eqn:coords} we deduce that $v=0$, $\eta$ is constant, and
\begin{equation}\label{eqn:NewEquilibria}
    u=Y_y \omega(Y-h-\eta) + \tilde cY_y - \omega(y-h) - c.
\end{equation}
Enforcing \eqref{eqn:lowermean}--\eqref{eqn:uppermean}, i.e., that the pseudofluxes are zero, yields a system of algebraic equations
\begin{subequations}\label{eqn:shearPseudoflux}
\begin{align}\label{eqn:shearPseudoflux:lower}
    0 &= -\tfrac 1 2 \omega_0(h+\eta)^2 + \tilde c (h+\eta) + \tfrac 1 2 \omega_0 h^2 - hc,\\
    \label{eqn:shearPseudoflux:upper}
    0&= \tfrac 1 2 \omega_1 (1-h-\eta)^2 + \tilde c (1-h-\eta) - \tfrac 1 2 \omega_1 (1-h)^2 - c (1-h)
\end{align}
\end{subequations}
for $\eta \in (-h,1-h)$ and $\tilde c \in \R$. Either equation can be uniquely solved for $\tilde c$ as a function of $\eta$, and eliminating $\tilde{c}$ yields \eqref{eqn:cubicForEquilibriumEta} as desired.
\end{proof}

\subsection{The evolution equation}
We now reformulate \eqref{eqn:stream2} as an evolution equation in $x$. The first step is to algebraically solve \eqref{eqn:stream2} for the derivatives $u_x$, $v_x$, and $\eta_x$.
\begin{subequations}
\label{F}
  Solving the boundary condition \eqref{eqn:stream2:kinint2} for $\eta_x$ gives us
  \begin{align}
    \label{F:F3}
    \eta_x &=  \dfrac{(h+\eta)(1-(h+\eta)) v(x,h)}{h(1-h) (u(x,h) +c)}.
  \end{align}
  Solving \eqref{eqn:stream2:divfree2} for $u_x$, and substituting \eqref{F:F3} to eliminate $\eta_x$ we find
  \begin{align}
    \label{F:F1}
    u_x &=  \dfrac{  v(x,h)}{h(1-h) (u(x,h) +c)} \Big( (1-2y)(u + \omega(y-h) + c) + y(1-y) (u_y + \omega) \Big) - v_y,
  \end{align}
  while similarly rearranging \eqref{eqn:stream2:lap2} yields
  \begin{align}
    \label{F:F2}
    v_x &= \dfrac{2 \eta(\eta y - \eta h + h(1-h))^3  (u + \omega(y-h) + c) + (\eta y - \eta h + h(1-h))^4 (u_y + \omega)  }{h^2(1-h)^2(h+\eta)^2(1-(h+\eta))^2}  \nonumber \\
    & \qquad +\frac{ y (1-y)}{h(1-h)( u(x,h) + c)}v(x,h) v_y - \omega.
  \end{align}
\end{subequations}
We then abbreviate \eqref{F} as
\begin{align}\label{eqn:stream3}
  \frac{\partial}{\partial x}
  \begin{pmatrix}
    u \\ v \\ \eta 
  \end{pmatrix}
  &= \mathcal{F}(u,v,\eta;c).
\end{align}
Since the formula for $\mathcal F$ does not contain 
any $x$ derivatives, we can think of $x$ as being fixed, meaning $u$ and $v$ are functions of $y$ only, and $\eta$ is just a real number.

Another way of thinking of this change of viewpoint is that $\mathcal F$ is a map between Banach spaces of functions in $y$.
We now define these function spaces: let
\begin{equation}\label{eqn:defining function spaces}
\begin{aligned}
  \mathcal X&=\left\{ (u,v,\eta) \in L^2((0,1)) \times L^2((0,1)) \times \mathbb{R} \ \middle| \ \int_0^h u(y) \; dy= \int_h^1 u(y) \; dy= 0 \right\}\\
  \mathcal W&=\left\{ (u,v,\eta) \in H^1((0,1)) \times H_0^1((0,1)) \times \mathbb{R} \ \middle| \ \int_0^h u(y) \; dy = \int_h^1 u(y) \; dy= 0 \right\}\\
  \mathcal{U}&= \big\{ (u,v,\eta,c) \ \mid \ (u,v,\eta) \in \mathcal{W}, \ u(h) \neq -c, \ h(1-h)u(h) \neq -c(h+\eta)(1-h-\eta)    \big\}.
\end{aligned} 
\end{equation}
More plainly, $\mathcal{U}$ is the subset of $\mathcal{W} \times \R$ in which none of the denominators that appear in \eqref{F} vanish. We see that $\mathcal{U}$ is open and for any $c_* \neq 0$, $\mathcal{U}$ contains $(0,0,0,c_*)$.

\begin{lem}\label{lem:analytic}
    $\mathcal{F} \colon \mathcal U \to \mathcal X$ is analytic.
\begin{proof}
Fix $(u,v,\eta,c) \in \mathcal U$. Then $u,v \in H^1((0,1)) \subset C^0([0,1])$, and so the derivatives $u_y$ and $v_y$ are well-defined, as are the pointwise values $u(h)$, $v(0)$, $v(1)$, and $v(h)$. Comparing with \eqref{F}, we conclude that $\mathcal F(u,v,\eta;c)$ is a well-defined element of $L^2((0,1)) \times L^2((0,1)) \times \mathbb{R}$. To verify the integral conditions in the definition of $\mathcal X$, let $\mathcal F (u,v,\eta;c) = (f,g,\alpha)$. For the first integral condition, we see that
\begin{equation*} \label{eqn:XisCodomain}
\begin{aligned}
    \int_0^h f \; dy 
    &= \int_0^h \bigg( \frac{  v(h)\big( (1-2y)(u + \omega(y-h) + c) + y(1-y) (u_y + \omega) \big)}{h(1-h) (u(h) +c)}  - v_y \bigg) \; dy\\
    &= \frac{  v(h)}{h(1-h) (u(h) +c)} \int_0^h \dfrac{d}{dy} \Big(  y(1-y)(u + \omega(y-h) + c) \Big) \; dy - \int_0^h v_y  \; dy\\
    &= v(h)-v(h) =0.
\end{aligned}    
\end{equation*}
The second integral condition follows by a similar argument. 

Now we show analyticity. Pointwise multiplication by a bounded function, and differentiation are both linear and bounded from $\mathcal W$ to $\mathcal X$, therefore analytic. 
Similarly, the relevant trace maps are bounded and linear from $\mathcal U \to \C$.
As $\mathcal F$ is a composition of these analytic maps with appropriate rational functions, we conclude that it is also analytic.
\end{proof}
\end{lem}  

We now separate $\mathcal F$ into linear and non-linear parts around $(u,v,\eta;c)=(0,0,0;c_*)$ with $c=c_* \neq 0$, i.e., around a shear solution which does not stagnate on the interface. More precisely, we define a linear mapping $L$ and nonlinear remainder $\mathcal R$ by
\begin{equation}\label{eqn:defn of L and R}
    L = (D_{u,v,\eta} \mathcal{F})(0,0,0,c_*), \quad \mathcal R(u,v,\eta,\eps) = \mathcal{F}(u,v,\eta,c_*+\eps) - L(u,v,\eta).
\end{equation}
For the moment the value of $c_*$ is unspecified, but we will end up focusing on the case $c_*=h(1-h)$. Calculating $L$ explicitly, we find
\begin{equation*}\label{eqn:introducingLexplicitly}
    L \begin{pmatrix}
    u\\
    v\\
    \eta\\
  \end{pmatrix} = 
  \begin{pmatrix}
    p(y) v(h) - v_y \\
    u_y - c_* \eta p'(y) \\
    v(h)/c_*
  \end{pmatrix}
\end{equation*}  
where the coefficient function $p$ is defined by
\[
  p(y) = \frac{ (1-2y)(\omega(y-h) + c_*) + y(1-y) \omega }{h(1-h)c_*}.
\]
As $\mathcal F$ is analytic by Lemma~\ref{lem:analytic}, $L\colon \mathcal W \to \mathcal X$ is a bounded linear operator while $\mathcal R \colon \mathcal U \to \mathcal X$ is analytic.

\section{The centre manifold}\label{sec:theCentreManifold}
We now state a centre manifold theorem. This is a result due to Mielke \cite{Mielke:CMT}; see Theorem~3.3 of \cite{HaragusIooss:CMT}, and Theorem~2.1 of \cite{MielkeAlexander1991HaLF}.
\begin{subequations}
It considers the general differential equation
\begin{equation}\label{eqn:introducingw}
    \frac{dw}{dx} = Lw + \mathcal{R}(w,\eps).
\end{equation}
In our application, the variable $w(x,y)$ corresponds to $(u(x,y),v(x,y),\eta(x))$, and $\eps=c-c_*$. 

\begin{thm}[Centre manifold theorem]\label{thm:CMT}
Let $\mathcal{X}, \mathcal{W}$ be Hilbert spaces with $\mathcal W $ continuously embedded in  $\mathcal X$.
Suppose a bounded linear operator $L \colon \mathcal W \to \mathcal X$ has spectrum $\sigma(L)$, and satisfies the following three hypotheses:
\begin{enumerate}[label=\rm(\roman*)]
    \item \label{hypothesis1} The centre spectrum $\sigma_0 = \{ z \in \sigma(L) \mid \Re(z)=0 \}$ consists only of finitely many eigenvalues, all of which have finite algebraic multiplicity.
    \item \label{hypothesis2} There exist $R>0$, $C>0$, such that for all $k \in \R$ with $|k|>R$, we have that $L-ikI \colon \mathcal X \to \mathcal X$ is invertible, and satisfies \[|k| \lVert (u,v,\eta) \rVert_ \mathcal X  \leq C \lVert (L-ikI)(u,v,\eta) \rVert _ \mathcal X.\]
    \item \label{hypothesis3} There exists $\delta>0$ such that there are no $z \in \sigma(L)$ satisfying $0< |\Re(z)|<\delta$. 
\end{enumerate}
Suppose we also have a nonlinear function $\mathcal R$ and an integer $N \geq 2$, such that there exists a neighbourhood $\mathcal{U} \subset \mathcal{W} \times \R$ of $0$ such that $\mathcal R \in C^N(\mathcal{U},\mathcal{X})$, and that
\begin{equation*}
    \mathcal{R}(0,0)=0, \quad D_w \mathcal{R}(0,0)=0.
\end{equation*}
Define $\mathcal{E}_0\subset \mathcal{W}$ to be the generalised eigenspace corresponding to the purely imaginary eigenvalues of $L$. Let $P_0$ be a continuous projection onto $\mathcal{E}_0$ which commutes with $L$, and whose kernel is called $\mathcal{E}_h$.

Then there exists a map $\psi \in C^N (\mathcal{E}_0 \times \R ,\mathcal{E}_h)$ with
\begin{equation}\label{eqn:psiAt0}
    \psi(0,0)=0, \quad D_w \psi(0,0)=0,
\end{equation}
and a neighbourhood $\mathcal{V}_w \times \mathcal{V}_\eps$ of $(0,0)$ in $\mathcal{E}_0 \times \R$, such that for $\eps \in \mathcal{V}_\eps$, the manifold 
\begin{equation}
    \mathcal{M}_0(\eps) = \{ w_0 + \psi(w_0,\eps) \mid w_0 \in \mathcal{V}_w\}
\end{equation}  
has the following properties:
\begin{enumerate}[label=\rm(\alph*)]
    \item $\mathcal{M}_0(\eps)$ is locally invariant,  i.e., if $w$ is a solution of \eqref{eqn:introducingw} satisfying $w(0) \in \mathcal{M}_0(\eps) \cap \mathcal{V}_w$ and $w(x) \in \mathcal{V}_w$ for all $x \in [0, \hat{x} ]$, then $w(x) \in \mathcal{M}_0(\eps)$ for all $x \in [0, \hat{x} ]$.
    \item $\mathcal{M}_0(\eps)$ contains the set of bounded solutions of \eqref{eqn:introducingw} staying in $\mathcal{V}_w$ for all $x \in \R$, i.e., if $w$ is a solution of \eqref{eqn:introducingw} satisfying $w(x) \in \mathcal{V}_w$ for all $x\in \R$, then $w(0) \in \mathcal{M}_0(\eps)$.
    \item\label{thm:CMT:conclusion:reduced solutions} Suppose $w_0$ satisfies the \emph{reduced equation} 
    \begin{equation}\label{eqn:reducedEqn}
        \frac{ d w_0}{dx}=P_0\big(Lw_0 + L \psi(w_0,\eps) +\mathcal{R}(w_0+\psi(w_0,\eps),\eps)\big),
    \end{equation}
    and $w_0(x) \in \mathcal{V}_w \cap \mathcal{E}_0$ for all $x$. Then $w_0+\psi(w_0,\eps)$ is a solution to the full problem, i.e., satisfies \eqref{eqn:introducingw}.
\end{enumerate}
\end{thm}
\end{subequations}

\subsection{Analysis of the linearised operator}\label{sec:L}

We show that the linear operator $L$ defined in \eqref{eqn:defn of L and R} satisfies Hypotheses~\ref{hypothesis1}--\ref{hypothesis3} of Theorem~\ref{thm:CMT}.

Finding the spectrum of $L$ is made easier by the fact that $L$ is Fredholm. To show this, we first consider the slightly simpler operator $\tilde L$ given by 
\begin{align*}
  \tilde L\begin{pmatrix}
    u\\
    v\\
    \eta\\
  \end{pmatrix} &= L\begin{pmatrix}
    u\\
    v\\
    \eta\\
  \end{pmatrix} + 
  \begin{pmatrix}
    0 \\
    (c_*p'-1) \eta  \\
    0\\
  \end{pmatrix}.
\end{align*}
\begin{lem}
$\tilde L: \mathcal{W} \to \mathcal{X}$ is invertible.
\begin{proof}
First we show $\tilde L$ has trivial kernel. This is not a particularly complicated calculation, but variants of it will appear frequently, so this will serve as a simple example. 
Suppose $(u,v,\eta) \in \mathcal W$ is in the kernel of $\tilde L$. This means we seek continuous functions of $y$, namely $u$ and $v$, and a real number $\eta$ satisfying
\begin{subequations}\label{eqn:tildeL}
\begin{alignat}{1}
    \label{eqn:tildeL:tildeL1}
    p(y)v(h) - v_y &=0 \quad \text{ for } y \in (0,h) \cup (h,1)\\
    \label{eqn:tildeL:tildeL2}
    u_y-\eta&=0 \quad \text{ for } y \in (0,h) \cup (h,1)\\
    \label{eqn:tildeL:tildeL3}
    \frac{v(h)}{c_*}&=0\\
    \label{eqn:tildeL:tildeLbc1}
    v(0)=v(1)&=0\\
    \label{eqn:tildeL:tildeLbc2}
    \int_0^h u = \int_h^1 u &= 0, 
\end{alignat}
\end{subequations}
where \eqref{eqn:tildeL:tildeL1}--\eqref{eqn:tildeL:tildeL3} are just $\tilde L (u,v,\eta) = 0$ rewritten while \eqref{eqn:tildeL:tildeLbc1}--\eqref{eqn:tildeL:tildeLbc2} and the continuity of $u$ and $v$ are imposed by $(u,v,\eta) \in \mathcal W$.

The equations \eqref{eqn:tildeL:tildeL3} and \eqref{eqn:tildeL:tildeL1} imply that $v$ is constant on each of $(0,h)$ and $(h,1)$, so then \eqref{eqn:tildeL:tildeLbc1} gives that $v=0$.
Solving \eqref{eqn:tildeL:tildeL2} and appealing to the continuity of $u$ at $h$ yields
\begin{align*}
    u = \begin{cases}
            \eta y + C_0 &  0 \leq y \leq h\\
            \eta y + C_0 &  h \leq y \leq 1,
        \end{cases}
\end{align*}
for some constant $C_0$. Inserting this into the constraint \eqref{eqn:tildeL:tildeLbc2} yields $\eta = C_0 = 0$ and hence $u=0$, 
so that $(u,v,\eta) = (0,0,0)$ as desired.

Notice also that $\tilde L$ is of full range. In particular, the solution to $\tilde L (u,v,\eta) = (f,g,\alpha)$ can be shown by a direct calculation, which we leave to the reader, to be
\begin{align*}
u(y)&= \left(\frac 2 h \int_0^h G(\tilde y) \; d \tilde y - \frac{2}{1-h} \int_h^1 G(\tilde y) \; d \tilde y \right) (y-h) \\
& \qquad+ G(y) - \frac{1-h}{h} \int_0^h G(\tilde y) \; d\tilde y - \frac{h}{1-h} \int_h^1 G(\tilde y) \; d\tilde y \\
v(y)&= \int_0^y \alpha c_* p(\tilde y) -f(\tilde y) \; d \tilde y \\
\eta &= \frac 2 h \int_0^h G(\tilde y) \; d\tilde y - \frac{2}{1-h} \int_h^1 G(\tilde y) \; d\tilde y \qquad \text{where} \qquad G(y) = \int_0^y g(\tilde y) \; d\tilde y.
\end{align*}
It can also be shown straightforwardly that this $(u,v,\eta) \in \mathcal W$.
Therefore, $\tilde L $ is invertible as desired.
\end{proof}
\end{lem}
\begin{cor}
The spectrum of $L$ consists only of eigenvalues.
\begin{proof}
We have just shown $\tilde L$ is invertible, and hence in particular that it is Fredholm with index 0.
Therefore, since 
$ L -\tilde L $ has finite-dimensional range, and therefore is a compact operator, $L$ is also Fredholm of index 0.
Similarly, $L-zI$ is Fredholm of index 0 for all $z \in \C$.
Therefore $L-zI$ is injective if and only if it is surjective. This means the spectrum of $L$ is exactly those $z \in \C$ for which $L-zI$ has non-trivial kernel.
\end{proof}
\end{cor}

We are now ready to verify Hypothesis~\ref{hypothesis1} of Theorem~\ref{thm:CMT}. Writing $z=ik$, we ask for which $k\in \mathbb{C}$ the operator $L-ikI$ has non-trivial kernel. The answer turns out to be largely captured by the \emph{dispersion relation}:
\begin{equation}\label{eqn:DispersionRelation}
\begin{aligned}
    \frac{1}{c_*} &= \mathfrak d(k),\\
    \text{where } \quad \mathfrak{d}(k) &= k \big( \coth (kh) + \coth ((1-h)k) \big).
\end{aligned}
\end{equation}
However, not all eigenvalues are given by solutions to \eqref{eqn:DispersionRelation}. If
\begin{equation}\label{eqn:BadDispersionRelation}
    \sin (hz) = \sin ((1-h)z) = 0, \ z \neq 0,
\end{equation}
then $L-zI$ has non-trivial kernel, no matter the values of $c_*, \omega_0, \omega_1$. We are primarily concerned with \eqref{eqn:DispersionRelation} for two reasons. 
Firstly, \eqref{eqn:BadDispersionRelation} is very rarely satisfied; in particular, it requires $h$ to be rational, which generically is not true.
Secondly, even if there exists a $z_*$ satisfying \eqref{eqn:BadDispersionRelation}, this does not contradict our hypotheses. This is because this $z_*$ must be real with $|z_*|> \pi$, and the hypotheses are only concerned with eigenvalues on, or arbitrarily close to, the imaginary axis.

\begin{lem}\label{lem:NonzeroEvalues}
If $k \in \C \setminus \{0\}$ and $L-ikI$ has non-trivial kernel, then $k$ satisfies \eqref{eqn:DispersionRelation} or \eqref{eqn:BadDispersionRelation}. 
\begin{proof}
Suppose $(u,v,\eta) \neq 0$ solves $(L-ikI)(u,v,\eta)=0$.
Since $v_y=v(h)p(y)-iku$, and $u$ is in $H^1$, we have that $v$ is twice differentiable.
Substitution yields the second order ODE $v_{yy} - k^2v =0$, which has general solution
\begin{subequations}\label{eqn:periodicKernel}
\begin{align}\label{eqn:periodicKernel:v}
v= \begin{cases}
            A \sinh( ky) & \text{for} \quad 0 \leq y \leq h \\
            B \sinh( k(1-y)) & \text{for} \quad h \leq y \leq 1,
    \end{cases}
\end{align}
for some constants $A$ and $B$.
Therefore, inserting the formula for $v$ into the first component of $(L-ikI)(u,v,\eta)=0$ yields
\label{eqn:periodicKernel:u}\begin{align}
u= \begin{cases}
            c_*\eta p + iA \cosh (ky) & \quad \text{for} \quad 0 \leq y \leq h \\
            c_*\eta p - iB \cosh (k(1-y)) & \quad \text{for} \quad h \leq y \leq 1.
    \end{cases}
\end{align}
\end{subequations}
The conditions on $v(h)$ and the continuity of $u$ give the equations
\begin{subequations}\label{eqn:eignevalueBC}
\begin{alignat}{1}
    \label{eqn:eignevalueBC:A} 
    A \sinh( kh) &= ikc_*\eta\\
    \label{eqn:eignevalueBC:B}
    B \sinh( k(1-h)) &= ikc_*\eta\\ \label{eqn:eignevalueBC:dispersion}
    A \cosh( kh) + B \cosh( k(1-h)) &= i \eta .
\end{alignat}
\end{subequations}

Suppose first that $\eta=0$.
By \eqref{eqn:eignevalueBC}, $A=0$ if and only if $B=0$. Since we want a non-trivial solution, we must have $A \neq 0, B \neq 0$. 
Therefore $ikh=n\pi$ and $ik(1-h) = m\pi$ for some $m,n \in \N$, meaning $h$ must be rational, and $ik=(m+n)\pi$. By assumption $k \neq 0$, so we infer $m+n \neq 0$. 
These, with \eqref{eqn:eignevalueBC:dispersion} means that $(-1)^n A + (-1)^m B=0$. Together these give us that
\begin{align*}
v= A \sinh (ky) \quad \text{for } y \in [0,1].
\end{align*}
Therefore, we have a one-dimensional kernel spanned by
\begin{align*}
u&= \cos(iky)\qquad v= -\sin(iky)\qquad \eta =0.
\end{align*}
These are precisely the eigenvalues which correspond to \eqref{eqn:BadDispersionRelation}, rather than \eqref{eqn:DispersionRelation}.

Now suppose $\eta \neq 0$, and $(u,v,\eta) \neq 0$ solves $(L-ikI)(u,v,\eta)=0$. Examining \eqref{eqn:eignevalueBC:A} and \eqref{eqn:eignevalueBC:B}, we see that $\sinh(kh)$ and $\sinh(k(1-h))$ are both non-zero, so \eqref{eqn:eignevalueBC:A} and \eqref{eqn:eignevalueBC:B} can be solved for $A$ and $B$. Subsituting these values into \eqref{eqn:eignevalueBC:dispersion} yields 
 \eqref{eqn:DispersionRelation}.
\end{proof}
\end{lem}

It is interesting to note that $\mathfrak d$ is a meromorphic function on $\C$ in $k$, with a removable singularity at $k=0$. Using this to evaluate \eqref{eqn:DispersionRelation} at $k=0$, and then solving for $c_*$, yields $c_*=h(1-h)$.
Notice that $h(1-h)$ is never 0, so it is a valid value for $c_*$.

\begin{lem} \label{lem:0evalue}
$L$ has an eigenvalue of $0$ if and only if $c_*=h(1-h)$. Furthermore, in this case, $0$ is the only purely imaginary eigenvalue, and has algebraic multiplicity two.
\begin{proof}
We seek solutions to $L(u,v,\eta)=0$. It can be easily shown using the equations that come directly from $L(u,v,\eta)=0$, the conditions on $v$, and the integral conditions on $u$, that if a non-trivial kernel exists, it must be spanned by a vector of the form
\begin{align}\label{eqn:ustar}
  u=u_*&= \begin{cases}
    c_*  \left(p(y) - h^{-1} \right) & \quad  \text{for} \quad 0 \leq y \leq h \\
            c_*  \left(p(y) + (1-h)^{-1} \right) & \quad \text{for} \quad h \leq y \leq 1
      \end{cases}\\
    v&=0\\
  \eta &=1.
\end{align}
The choice $c_*=h(1-h)$ then guarantees $p$ is such that $u$ is continuous at $h$, so that we indeed have a one-dimensional kernel. 

We now show that there are no other purely imaginary eigenvalues. 
Lemma~\ref{lem:NonzeroEvalues} means that this is equivalent to showing that \eqref{eqn:DispersionRelation} has no non-zero real roots. 
We differentiate $\mathfrak d$ with respect to $k$, and see
\begin{align*}
    \mathfrak d '(k) &= \frac{\sinh (2hk) - 2hk}{2 \sinh^2(hk)} +\frac{\sinh (2(1-h)k) - 2(1-h)k}{2 \sinh^2((1-h)k)}.
\end{align*}
Therefore $\mathfrak d (k)$ is strictly decreasing for $k<0$ and strictly increasing for $k>0$, so it achieves its minimum value only at $k=0$. Hence, for $k \in \R \setminus \{ 0 \}$, we have \[\mathfrak d (k) > \frac{1}{c_*}. \]
Therefore \eqref{eqn:DispersionRelation} has no nonzero real roots, thus $0$ is the only purely imaginary eigenvalue.

We now seek second order generalised eigenvectors, i.e., solutions to $L(u,v,\eta) = (u_*, 0, 1)$. Without loss of generality, we can take $\eta = 0$.
Therefore we need to solve
\begin{align*}
    c_*p - v_y&=u_*\\
    u_y &= 0\\
    v(0)&=v(1)=0\\
    v(h)&=c_*\\
    \int_0^h u \; dy &= \int_h^1 u \; dy=0,
\end{align*}
which has solution $(0,v_*,0)$ where
\begin{align*}
  v_*= \begin{cases}
    		c_* y h^{-1} & \quad \text{for} \quad 0 \leq y \leq h \\
            c_* (1-y)(1-h)^{-1} & \quad \text{for} \quad h \leq y \leq 1.
        \end{cases}
\end{align*} 

Performing a similar process to find a third order generalised eigenvector shows that none exist, so we conclude that when $c_*=h(1-h)$, $0$ is an eigenvector, and has algebraic multiplicity 2.
\end{proof}
\end{lem}

\begin{cor}
For $c_*=h(1-h)$, $L$ satisfies Hypothesis~\ref{hypothesis1} of Theorem~\ref{thm:CMT}.

\begin{proof}
We have shown that for this value of $c_*$, $L$ has one purely imaginary eigenvalue, and it has finite algebraic multiplicity. Furthermore, since $L$ is Fredholm, its spectrum consists only of eigenvalues, therefore $\sigma_0=\{0\}$. Thus, Hypothesis~\ref{hypothesis1} is satisfied.
\end{proof}
\end{cor}

We now verify Hypothesis~\ref{hypothesis2}. Since $\sigma_0$ is finite, the inverse $(L-ikI)^{-1}$, which we view as a linear operator $\mathcal X \to \mathcal X$, exists for sufficiently large $k$.

\begin{prop}
    For any value of $c_* \neq 0$, the operator $L$ satisfies Hypothesis~\ref{hypothesis2} of Theorem~\ref{thm:CMT}.
\begin{proof}
Suppose $(L-ikI)(u,v,\eta) = (f,g,\alpha)$.
For the remainder of this proof, $\lVert \, \cdot \, \rVert$ should be taken to mean the $L^2((0,1))$ norm, and $P=1+\lVert p \rVert$.

Our first step will be to find a bound on $|v(h)|$ which grows sub-linearly in $k$. We see that
\begin{align*}
v(h)^2 &= \int_0^h 2v v_y \; dy 
= -2ik \int_0^h uv \; dy - 2\int_0^h fv \; dy + 2v(h)\int_0^h pv \; dy, 
\end{align*}
and therefore,
\begin{align*}
|v(h)|^2 &\leq 2|v(h)|\lVert p \rVert \lVert v \rVert  + 2|k| \lVert u \rVert \lVert v \rVert + 2\lVert f \rVert \lVert v \rVert. 
\end{align*}
Thinking of this as a quadratic in $|v(h)|$, then applying Cauchy--Schwarz gives us
\begin{align}
\nonumber |v(h)| &\leq  \lVert p \rVert \lVert v \rVert + \sqrt{\lVert p \rVert^2 \lVert v \rVert^2 + 2 \lVert ku \rVert \lVert v \rVert + 2 \lVert f \rVert \lVert v \rVert} \\
\nonumber&
\leq  2\lVert p \rVert \lVert v \rVert + \sqrt{2 \lVert ku \rVert \lVert v \rVert} +  \lVert f \rVert +  \lVert v \rVert
\leq  2P \lVert v \rVert + \sqrt{2|k|} \lVert (u,v) \rVert  +  \lVert f \rVert\\
\label{eqn:bound v(h)}&\leq \left(2P + \sqrt{2|k|} \right) \lVert (u,v) \rVert  + \lVert f \rVert,
\end{align}
which is a bound of the desired form.

Now we multiply the first component of $(L-ikI)(u,v,\eta) = (f,g,\alpha)$ by $\bar{u}$, the complex conjugate of the second by $v$, and subtract to get
\begin{align*}
\bar{u} v_y + \bar{u}_y v + ik( \lvert u \rvert ^2 +  \lvert v \rvert ^2) -c_*  \bar{\eta} p' v - v(h) p \bar{u} &= -f \bar{u} + \bar{g} v.
\end{align*}
Integrating over $[0,1]$ yields
\begin{align*}
\int_0^1  (\bar{u} v)_y \; dy + ik  \lVert (u,v) \rVert^2  &=  \int_0^1 -\bar{f} u + \bar{g} v + v(h)p \bar{u} + c_* \bar{\eta} p' v  \; dy.
\end{align*}
The boundary conditions on $v$ mean that $\bar{u}(0)v(0)=\bar{u}(1)v(1)=0$, and so using this and Cauchy--Schwarz,
\begin{align*}
|k| \lVert (u,v) \rVert^2 & \leq \lVert (u,v) \rVert \lVert (f,g) \rVert + |v(h)| \lVert p \rVert \lVert u \rVert + |\eta| |c_*|  \lVert p' \rVert \lVert v \rVert,
\end{align*}
so by \eqref{eqn:bound v(h)},
\begin{align*}
\left( |k| - \lVert p \rVert \sqrt{2|k|}-2P^2 \right) \lVert (u,v) \rVert^2 & \leq \lVert (u,v) \rVert \lVert (f,g) \rVert  + \lVert p \rVert \lVert u \rVert \lVert f \rVert + \lvert \eta \rvert |c_*|  \lVert p' \rVert \lVert v \rVert.
\end{align*}
Dividing through by $\lVert (u,v) \rVert$ gives us
\begin{align*}
\left( |k| - \lVert p \rVert \sqrt{2|k|}-2P^2 \right) \lVert (u,v) \rVert & \leq P \lVert (f,g) \rVert  + \lvert \eta \rvert |c_*|  \lVert p' \rVert,
\end{align*}
then adding $\left( |k| - |c_*| \lVert p' \rVert \right)|\eta|$ to both sides yields
\begin{align*}
\left( |k| - \lVert p \rVert \sqrt{2|k|}-2P^2 \right) \lVert (u,v) \rVert + \left( |k| - |c_*| \lVert p' \rVert \right)|\eta| & \leq P \lVert (f,g) \rVert  + |k\eta|.
\end{align*}
Then, applying Cauchy--Schwarz shows that
\begin{align*}
\frac 12 \left( |k| - \lVert p \rVert \sqrt{2|k|}-2P^2 - |c_*| \lVert p' \rVert \right) \lVert (u,v,\eta) \rVert & \leq P \lVert (f,g) \rVert  + |k\eta|.
\end{align*}
The third component of $(L-ikI)(u,v,\eta)=(f,g,\alpha)$ tells us that $v(h)/c_*-ik\eta = \alpha$. Applying this, then \eqref{eqn:bound v(h)} to the final term of the previous line yields
\begin{align*}
\frac 12 \left( |k| - \lVert p \rVert \sqrt{2|k|}-2P^2 - |c_*| \lVert p' \rVert \right) \lVert (u,v,\eta) \rVert & \leq P \lVert (f,g) \rVert  +  |\alpha|+|c_*|^{-1}|v(h)|\\
& \leq P \lVert (f,g) \rVert  +  |\alpha|+ |c_*|^{-1} \lVert f \rVert\\
& \quad +|c_*|^{-1}\left(2P + \sqrt{2|k|} \right) \lVert (u,v) \rVert .
\end{align*}
Finally, rearranging, we see
\begin{align*}
\lVert (u,v,\eta) \rVert & \leq \frac{ 2(P+|c_*|^{-1}) \lVert (f,g) \rVert  +  2|\alpha|}{ |k| - \left(\lVert p \rVert + 2 |c_*|^{-1} \right)\sqrt{2|k|}-2P^2 - |c_*| \lVert p' \rVert - 4 P |c_*|^{-1}}.
\end{align*}
For sufficiently large $|k|$, the denominator is larger than $\frac 12 |k|$.
Therefore there exists $R>0$ such that for $|k|>R$,
\begin{align*}
|k| \lVert (u,v,\eta) \rVert_ \mathcal X & \leq C \lVert (f,g,\alpha) \rVert _ \mathcal X
\end{align*}
as required, where the constant $C$ depends only on $R$, $h$, $\omega_0$, $\omega_1$, and $c_*$. 
\end{proof}
\end{prop}

We now verify Hypothesis~\ref{hypothesis3}, and show a spectral gap around the imaginary axis. 
This could in principle be done by only studying the solutions to  \eqref{eqn:DispersionRelation}. However, we 
use an easier proof which makes use of Hypothesis~\ref{hypothesis2} as well.

\begin{prop}
There exists $\delta>0$ such that there are no $z$ in the spectrum of $L$ satisfying $0<|\Re(z)|<\delta$.

\begin{proof}
We show this by contradiction. Suppose this spectral gap does not exist. Hence, there must exist a sequence $(z_n)$ in the spectrum such that $\Re(z_n) \to 0$.
Define 
\[ \mathcal S = \sigma(L) \cap \{ z \in \C \mid - \pi \leq \Re(z) \leq \pi \} ,\]
where $\sigma(L)$ is the spectrum of $L$. None of the $z \in \mathcal S$ can satisfy \eqref{eqn:BadDispersionRelation}, so for all $z \in \mathcal S$, we have that \eqref{eqn:DispersionRelation} holds, i.e., that $\mathfrak{d}(-iz) = c_*^{-1}$. Since $\mathfrak{d}(-iz)$ is a meromorphic function in $z$, $\mathcal S$ is an isolated set.
In other words, $(z_n)$ is a sequence taking values in an isolated set, with $\Re(z_n) \to 0$, so we conclude $ |\Im(z_n)| \to \infty$.

However, by Hypothesis~\ref{hypothesis2} there exist constants $C$, $R$ such that  for all $z \in i\R$ with $|z|>R$, we have 
\[|z| \lVert (u,v,\eta) \rVert_ \mathcal{X}  \leq C \lVert (L-zI)(u,v,\eta) \rVert_ \mathcal X , \]
and $L-zI$ invertible. Pick $n_*$ such that for all $n>n_*$, we have that $ |\Im(z_n)| > R$ and $|\Re(z_n)|<R/C$. Let $\tilde{z}_n = i \Im (z_n)$. 
Then $(z_n-\tilde{z}_n)(L-\tilde{z}_nI)^{-1} \colon \mathcal{X} \to \mathcal{X}$ has operator norm strictly less than $1$. 
This means $I - (z_n-\tilde{z}_n)(L-\tilde{z}_nI)^{-1} \colon \mathcal X \to \mathcal X$ is injective, so $I - (z_n-\tilde{z}_n)(L-\tilde{z}_nI)^{-1} \colon \mathcal{W} \to \mathcal{W}$ is injective. Hence 
\[L-z_nI = (L-\tilde{z}_nI) (I-(L-\tilde{z}_nI)^{-1}(z_n-\tilde{z}_n)I)\] is also injective.
Therefore, for all $n>n_*$, we have that $z_n$ is in the resolvent set, but we assumed it was in the spectrum, and so arrive at our contradiction.
\end{proof}
\end{prop}

We have now verified Hypotheses~\ref{hypothesis1}--\ref{hypothesis3} of Theorem~\ref{thm:CMT} in the case $c_* = h(1-h)$, and so proceed for the rest of the paper with $c_* = h(1-h)$. This specific value of $c_*$ is of particular interest, as it corresponds to $k=0$, so we might expect to find waves of infinite period, i.e., solitary waves.

The final step in our analysis of $L$ is to seek the projection $P_0$ from Theorem~\ref{thm:CMT}.  
Notice that because 0 is the only purely imaginary element of the spectrum, $\mathcal{E}_0$ is the generalised kernel of $L$. 
Lemma~\ref{lem:0evalue} gives us a basis $\{\xi_0,\xi_1\}$ of $\mathcal{E}_0$, where 
\begin{align}
    \label{eqn:xi}
    \xi_0 = (u_*,0,1) \qquad \xi_1 = (0,v_*,0).
\end{align}
Thus we can write $P_0$ as
\begin{align}\label{eqn:form of P0}
    P_0(u,v,\eta) &= A(u,v,\eta) \xi_0 + B(u,v,\eta)\xi_1,
\end{align}     
where $A,B \colon \mathcal X \to \R $ are bounded linear functions. Calculating
\begin{align*}
    LP_0= 
    \begin{pmatrix}
        u_* \\ 0\\ 1
    \end{pmatrix} 
    B
    \quad \text{and} \quad 
    P_0L=  
    \begin{pmatrix}
      u_* \\ 0\\ 1
    \end{pmatrix}
    AL
    + 
    \begin{pmatrix}
      0 \\ v_*\\ 0
    \end{pmatrix}
    BL,
\end{align*}
we see that $L$ and $P_0$ commute if and only if $A \circ L = B$ and $B \circ L =0$. 

\begin{prop}
    \begin{subequations}\label{eqn:formula for projection}
        The projection $P_0$ from Theorem~\ref{thm:CMT} is given by \eqref{eqn:form of P0}, where
        \begin{align}
            \label{eqn:formula for projection:A} A(u,v,\eta)&=\frac{3}{2h^2(1-h)} \int_0^h y^2u(y) \; dy - \frac{3}{2h(1-h)^2} \int_h^1 (1-y)^2 u(y) \; dy\\
        	   \nonumber & \quad \quad + \frac {( 8(1-h)^4 - 15(1-h)^3  + 5(1-h))\omega_1 - ( 8h^4 - 15h^3 + 5h)\omega_0}{20h^2(1 - h)^2 }\eta\\
            \label{eqn:formula for projection:B} B(u,v,\eta)&= \frac{3}{h^2(1-h)}\int_0^h y v(y) \; dy + \frac{3}{h(1-h)^2}\int_h^1 (1-y) v(y)\; dy,
        \end{align}
        and $\xi_0 = (u_*,0,1)$ and $\xi_1 = (0,v_*,0)$, as defined in \eqref{eqn:xi}.
    \end{subequations}
\begin{proof}
This can be verified by brute force. However, we now show a little of the construction, in order to give some intuition about why the projection is of this form. 
We first seek $B$. Notice that $L(u,v,\eta) = (f,g,a)$ implies that 
$u-h(1-h) \eta p = G+ \text{constant}$, where $G$ is any primitive of $g$. This implies
\begin{equation}\label{eqn:motivationforB}
    \frac{1}{h}\int_0^h G(y) \; dy - \frac{1}{1-h} \int_h^1 G(y) \; dy = 0.
\end{equation} 
Now construct $B$ such that it satisfies $B \circ L =0$. 
Motivated by \eqref{eqn:motivationforB}, let $V$ be a primitive of $v$, and define 
\begin{align*}
    B(u,v,\eta) &= \frac{3}{h(1-h)}\left(- \frac{1}{h}\int_0^h V \; dy + \frac{1}{1-h} \int_h^1 V \; dy \right)\\
    &=\frac{3}{h(1-h)}\left( - \frac{1}{h}\int_0^h (h-y)v \; dy + \frac{1}{1-h}\int_h^1 (1-y)v \; dy + \int_0^h v \; dy\right)\\
    &= \frac{3}{h^2(1-h)}\int_0^h y v \; dy + \frac{3}{h(1-h)^2}\int_h^1 (1-y) v\; dy .
\end{align*} 
If $L$ had an inverse, solving $A \circ L = B$ for $A$ would give a formula for $A$. Doing exactly this is impossible, as $L$ is not invertible, but we can use the fact that $B$ depends only on $v$ to do something very similar.  Notice that $L(u,v,\eta) = (f,g,a)$ implies that 
$h(1-h)\alpha p - f =v_y$. This motivates the definition
\begin{align*}
	A(u,v,\eta) &= B \left(u, \int_h^y h(1-h) \eta p(\tilde{y}) - u(\tilde{y}) \; d \tilde{y} + h(1-h) \eta, \eta \right)\\	
	&= -\frac{3}{2h^2(1-h)} \int_0^h y^2(h(1-h) \eta p(y) - u(y)) \; dy \\
	& \quad \quad+ \frac{3}{2h(1-h)^2} \int_h^1 (1-y)^2(h(1-h)\eta p(y) - u(y)) \; dy +\frac{3}{2} \eta \\
	&=\frac{3}{2h^2(1-h)} \int_0^h y^2u(y) \; dy - \frac{3}{2h(1-h)^2} \int_h^1 (1-y)^2 u(y) \; dy\\
	& \quad \quad + \frac {( 8(1-h)^4 - 15(1-h)^3  + 5(1-h))\omega_1 - ( 8h^4 - 15h^3 + 5h)\omega_0}{20h^2(1 - h)^2 }\eta .
\end{align*}
Notice that
\begin{align*}
	A(L(u,v,\eta)) &= B \left( v(h)p - v_y, \int_h^y v(h)p - v(h)p + v_y \; d \tilde {y} + v(h), \frac{v(h)}{h(1-h)} \right)\\
	&=B(u,v,\eta).
 \end{align*}
Therefore $A$ and $B$ exhibit the required behaviour when composed with $L$. It remains only to check that $P_0$ is indeed a projection onto the generalised kernel of $L$. We see that
\[ A(u_*, 0,1) = B(0,v_*,0) = 1,\]
and
\[ B(u_*, 0,1) = A(0,v_*,0) = 0,\]
so we are done.     
\end{proof}
\end{prop}

\subsection{Applying the centre manifold theorem}\label{sec:CMT}

We quickly verify the hypotheses of Theorem~\ref{thm:CMT} on the non-linear operator $\mathcal R$. 
First, fix an integer $N \geq 4$ which will not change for the rest of this paper. 
Although Theorem~\ref{thm:CMT} only needs this $N$ to be greater than or equal to $2$, we need it to be greater than or equal to $4$ to make some arguments about the regularity of the solutions we find. 
See, for example, \eqref{eqn:axandbx:bx} and Appendix~\ref{sec:appendix}.

The remainder function $\mathcal R$ defined in \eqref{eqn:defn of L and R} is analytic, so is indeed in $C^N(\mathcal{U}, \mathcal X)$, where $\mathcal{U}$ is the open set defined in \eqref{eqn:defining function spaces}. Furthermore, we obtained $\mathcal R$ as a remainder after linearising in $(u,v,\eta)$, so by construction it satisfies 
\[\mathcal R (0,0,0,0)=0, \quad D_{(u,v,\eta)}\mathcal R (0,0,0,0)=0 .\]

We are now ready to apply Theorem~\ref{thm:CMT} and find solutions to \eqref{eqn:stream2} on the centre manifold.
Let $\eps = c-c_* = c - h(1-h)$, and let $w=(u,v,\eta)$.
By Theorem~\ref{thm:CMT}, there exists $\delta>0$, $\psi \in C^N(\mathcal{E}_0 \times \R, \mathcal{E}_h )$ such that
for all $\eps \in (-\delta,\delta)$, if $w$ satisfies
\begin{equation}\label{eqn:wIsBddSoln}
    \frac{dw}{dx} = Lw + \mathcal{R}(w,\eps), \quad \sup_x \lVert w(x, \, \cdot \, ) - (\eps,0,0) \rVert_{\mathcal{W}} < \delta    
\end{equation}
then it must be of the of the form
\begin{align}\label{eqn:formOfSolnOnCMT}
     w(x,y) = a(x) \xi_0(y) + b(x) \xi_1(y) + \psi \big(a(x) \xi_0 + b(x) \xi_1, \eps \big)(y),
\end{align}
where $\xi_0,\xi_1$ are defined in \eqref{eqn:xi} and $a$ and $b$ are scalar functions. 

We now introduce and clarify some points of notation. Firstly, since $\mathcal{E}_0 \times \R$, the domain of $\psi$, is a 3 dimensional space, we abuse notation slightly, and sometimes consider $\psi$ to be a function on $\R^3$  given by
\[ \psi(a,b,\eps) = \psi(a \xi_0 + b \xi_1, \eps ). \]
Secondly, since $\psi$ takes values in $\mathcal W$, we denote its $u$, $v$ and $\eta$ components as $e_1 \cdot \psi$, $e_2 \cdot \psi$, and $e_3 \cdot \psi$ respectively. 
This means \eqref{eqn:formOfSolnOnCMT} can be written completely equivalently as
\begin{align}
\label{eqn:formOfSolnOnCMTbig}
\begin{pmatrix}
  u(x,y)\\
  v(x,y)\\
  \eta(x)
\end{pmatrix} = a(x) \begin{pmatrix}
  u_*(y)\\
  0\\
  1
\end{pmatrix} + b(x) \begin{pmatrix}
  0\\
  v_*(y)\\
  0
\end{pmatrix} + \begin{pmatrix}
  e_1 \cdot \psi(a(x),b(x),\eps)\\
  e_2 \cdot \psi(a(x),b(x),\eps)\\
  e_3 \cdot \psi(a(x),b(x),\eps)\\
\end{pmatrix}(y).
\end{align} 

Reversibility immediately gives a symmetry on $\psi$. Recall we defined reversibility in \eqref{eqn:UVreversibility}. Using Theorem 3.15 of \cite{HaragusIooss:CMT}, we see that
\begin{equation}\label{eqn:reversibilityOfPsi}
    \psi(a,-b,\eps) = \begin{pmatrix}
  e_1 \cdot \psi(a,b,\eps)\\
  -e_2 \cdot \psi(a,b,\eps)\\
  e_3 \cdot \psi(a,b,\eps)
\end{pmatrix},
\end{equation}  
i.e., that the first and third components of $\psi$ are even in $b$, and the second component is odd in $b$.

In order for a solution of the form in \eqref{eqn:formOfSolnOnCMT} to exist, $a$ and $b$ must satisfy a differential equation, which is found using $P_0$.
Inserting \eqref{eqn:formOfSolnOnCMT} into \eqref{eqn:wIsBddSoln} yields
\begin{equation} \label{eqn:wIsSolution}
\begin{aligned}
    L(w) + \mathcal{R}(w,\eps) &= w_x 
    = a_x \xi_0 + b_x \xi_1 + (D\psi(w,\eps))(a_x \xi_0 + b_x \xi_1, \eps).
\end{aligned}
\end{equation}
We now apply $P_0$ to find the reduced equation \eqref{eqn:reducedEqn}. The composition $P_0\circ \psi = 0$, therefore $P_0 \circ D\psi = 0$. By construction, $P_0$ commutes with $L$, therefore applying $P_0$ to \eqref{eqn:wIsSolution} gives
\begin{equation}
\begin{aligned}\label{eqn:P0w} a_x \xi_0 + b_x \xi_1 &=L(a \xi_0 + b \xi_1 + P_0\psi(a,b,\eps)) + P_0\mathcal R (w,\eps)
= b \xi_0 + P_0\mathcal R (w,\eps)\\
&= b \xi_0 + P_0\mathcal R(a\xi_0 + b\xi_1 + \psi(a,b,\eps),\eps).
\end{aligned}
\end{equation}
We can rewrite \eqref{eqn:P0w} more compactly as
\begin{equation}\label{eqn:directODE}
(a_x,b_x) = F(a,b,\eps).
\end{equation}

Let $e_1\cdot F$ and $e_2\cdot F$ denote the first and second components of $F$ respectively. 
Since $F$ is a composition of $C^N$ functions, it follows that $F \in C^N(\R^3, \R^2)$. 
Note that $F$ inherits the reversibility of $\psi$ shown in \eqref{eqn:reversibilityOfPsi}. 
Specifically, $e_1 \cdot F$ is odd in $b$, and $e_2 \cdot F$ is even in $b$.
This means that if $a(x)$ and $b(x)$ solve \eqref{eqn:directODE} and we define 
\[ \check a(x) = a(-x), \quad \check b(x) = - b(-x), \]
then $\check a(x)$ and $\check b(x)$ also  satisfy \eqref{eqn:directODE}. 

We now find expansions for $\psi$ and $F$.
Since for all $\eps$, the zero function $u=v=\eta=0$ is a solution of \eqref{eqn:stream2}, Theorem~\ref{thm:CMT} tells us that $\psi(0,0,\eps)=0$. 
Since $\psi_a(0,0,0)=\psi_b(0,0,0)=0$, we have
\begin{subequations}
\begin{equation}\label{eqn:orderOfpsi}
    \psi(a,b,\eps)=\mathcal{O}((|a| + |b|)(|a|+|b| + |\eps|)).
\end{equation}
Therefore,
\begin{equation}\label{eqn:orderOfpsi:remainder}
\begin{aligned}
    P_0 \mathcal R (a\xi_0 + b\xi_1 + \psi(a,b,\eps),\eps ) &=P_0  \mathcal R \big(a\xi_0 + b\xi_1 + \mathcal{O}((|a| + |b|)(|a|+|b| + |\eps|) ),\eps \big) \\
    &=P_0\mathcal R(a\xi_0 + b\xi_1,\eps) +\mathcal{O}\big((|a| + |b|)(a^2+b^2 + \eps^2) \big).
\end{aligned}
\end{equation}
\end{subequations}
Notice that we know $P_0\mathcal R(a\xi_0 + b\xi_1,\eps)$ explicitly. Substituting \eqref{eqn:orderOfpsi:remainder} into \eqref{eqn:P0w}, we obtain an expansion for $F$. 
However, appealing to parity can sharpen the order of the error term. 
Since $e_1 \cdot F$ is odd in $b$, and $e_2 \cdot F$ is even in $b$, we deduce
\begin{subequations}\label{eqn:axandbx}
\begin{align}
    \label{eqn:axandbx:ax}
    e_1\cdot F &= b(1+\mathcal{O}(|a|+|b| + |\eps|))\\
    \label{eqn:axandbx:bx}
    e_2 \cdot F&= \frac{3}{h^2(1-h)^2}\eps a + \frac{3 \theta}{2h^2(1-h)^2} a^2 +\mathcal{O}\big( b^2 + (|a| + b^2)(a^2+b^2 + \eps^2) \big),
\end{align}    
\end{subequations}
where recall $\theta$ was defined in \eqref{eqn:defineTheta}.
Notice that these really are the dominating terms, in that their coefficients are non-zero. The terms written here can be found by examining $P_0 \mathcal{R}(a x_0 + b \xi_1)$. Higher order terms can be found explicitly by a recursive method, but for our purposes this is unnecessary.

We now solve \eqref{eqn:directODE} for $a$ and $b$. 
For fixed $\eps$, by Picard--Lindel\"of, \eqref{eqn:directODE} with initial condition $(a(0),b(0))=(a_*,0)$ has a unique solution. By reversibility, this solution curve in the $a,b$ plane is unchanged under reflection about the $a$ axis. 
In other words, for fixed $\eps$, solutions to \eqref{eqn:directODE} that cross the $a$ axis are symmetric about the $a$ axis. This motivates us to focus on solutions where $a$ is even, and $b$ is odd.

We now rescale with the aim of eventually eliminating higher order terms, and proceed having picked $\eps$ to be strictly positive. Let 
\begin{equation}\label{eqn:RescaledCoords}
\tilde{x} = \frac{\sqrt{3 \eps}}{h(1-h)}x, \quad  a(x) = \frac{2\eps}{\theta} \tilde{a}(\tilde{x}), \quad b(x) = \frac{2 \sqrt 3 \eps^{\frac{3}{2}}}{h(1-h)\theta} \tilde{b}(\tilde{x}).    
\end{equation}
When rescaling, we have that $\tilde a = \mathcal O (\eps^{-1})a$, but $\tilde b = \mathcal O (\eps^{-\frac 3 2})b$. 
For conciseness, we write the rescaled version of \eqref{eqn:directODE}
\begin{equation}\label{eqn:rescaledDirectODE}
\begin{aligned}
    (\tilde {a}_{\tilde{x}},\tilde {b}_{\tilde{x}})= \tilde{F}(\tilde a, \tilde b, \eps).
\end{aligned}
\end{equation}
Notice that $\tilde F$ satisfies
\begin{equation}\label{eqn:ordersOfTildeF}
\tilde F (\tilde a , \tilde b , \eps)=\                                       \begin{pmatrix}
        \tilde {b}(1+\mathcal{O}(\eps) )\\
        \tilde {a}+ \tilde{a}^2+\mathcal{O}\big( \eps (|\tilde a|+|\tilde b|^2) \big)
    \end{pmatrix}.
\end{equation}     
When $\eps=0$, \eqref{eqn:ordersOfTildeF} is exactly the KdV equation for travelling waves, a nonlinear equation. We expect the solution to persist as $\eps$ moves away from $0$, and describe such a solution as \emph{weakly nonlinear}.
There are several ways to show that the solutions do indeed persist, for example, by considering the stable and unstable manifolds of $\tilde F$ at the fixed point at the origin. 
Instead, we adapt an argument of Kirchg\"assner \cite{Kirchgassner:bifurcation}. 
\begin{prop}\label{thm:aepsbepsExist}
    For all sufficiently small $\eps \geq 0$, \eqref{eqn:rescaledDirectODE} has a solution $(\tilde{a}^\eps,\tilde{b}^\eps)$ with the following properties.
    \begin{enumerate}[label=\rm(\alph*)]
        \item $\tilde a^\eps$ is even and $\tilde b^\eps$ is odd.
        \item $(\tilde a^\eps,\tilde b^\eps)$ is homoclinic to 0.
        \item $(\tilde a^\eps,\tilde b^\eps) \in C^N(\R)$.
        \item  The explicit formula for $(\tilde{a}^0,\tilde{b}^0)$, is given by
\begin{equation}\label{eqn:atilde0btilde0}
    \tilde a^0( \tilde x) = -\tfrac 32 \sech^2 \left(\tfrac 12 \tilde x \right), \quad \tilde b^0 (\tilde x) = \tfrac 32 \tanh \left(\tfrac 12 \tilde x\right) \sech^2 \left(\tfrac 12 \tilde x\right).
\end{equation}
        \item The map $\eps \mapsto \tilde{a}^\eps$ is in $C^{N-1}(\R,C^2(\R))$.
        \item\label{thm:atildebtildeLipschitz} $(\tilde a^\eps,\tilde b^\eps) =(\tilde a^0,\tilde b^0) + \mathcal O (\eps) $, where the $\mathcal O$ is with respect to the $C^1(\R)$ norm.
    \end{enumerate}
\begin{proof}
    We argue as in Proposition~5.1 of  \cite{Kirchgassner:bifurcation}. For more details, see Appendix~\ref{sec:appendix}. 
\end{proof}
\end{prop}
\begin{rk}\label{thm:uniqe0oftildeb}
    Since the phase portrait of $\tilde F$ is symmetric about the line $\tilde b =0$, the only zeroes of $\tilde b^\eps$ are at $\tilde x =0$.
    Indeed, if $\tilde b^\eps$ vanished at two distinct points, then $\tilde b^\eps$ would be periodic rather than homoclinic.
    The oddness of $\tilde b^\eps$ implies that $\tilde b ^\eps(0) =0$, therefore this zero is unique.
\end{rk}

\subsection{Regularity and estimates}

We now undo the scaling from \eqref{eqn:RescaledCoords} to find functions $a^\eps(x)$, $b^\eps(x)$ satisfying \eqref{eqn:directODE}. By conclusion~\ref{thm:CMT:conclusion:reduced solutions} of Theorem~\ref{thm:CMT}, inserting these into the reduction function does indeed give a solution to \eqref{eqn:stream2}.
Then undoing the coordinate change \eqref{eqn:coords} gives us solutions to \eqref{eqn:stream}.

We explicitly find the leading order terms in $C^0(\R, H^1((0,1)))$, and show that the solution is small in $C^1(\R, H^1((0,1)))$.
\begin{prop}\label{thm:uvetaapprox}
    For all sufficiently small $\eps\geq0$, there exist functions $u^\eps(x,y)$, $v^\eps(x,y)$, $\eta^\eps(x)$ such that:
    \begin{enumerate}[label=\rm(\alph*)]
        \item $(u^\eps,v^\eps,\eta^\eps; h(1-h)+\eps)$ solves \eqref{eqn:stream2}.
        \item $u^\eps$ and $\eta^\eps$ are even in $x$, $v^\eps$ is odd in $x$.
        \item\label{thm:uvetaapprox:leadingTerms} In $C^0(\R, H^1((0,1)))$, we have 
            \begin{subequations}\label{eqn:leadingOrderSolnuveta}
                \begin{align}
                    u^\eps(x,y) &= -\frac{3\eps}{\theta} \sech^2 \Big(\frac{\sqrt{3\eps}}{2h(1-h)} x\Big)u_*(y) + \mathcal{O}(\eps^2)\\
                    v^\eps(x,y) &= \frac{3 \sqrt 3 \eps^{\frac{3}{2}}}{h(1-h)\theta} \tanh \Big(\frac{\sqrt{3\eps}}{2h(1-h)} x\Big) \sech^2 \Big(\frac{\sqrt{3\eps}}{2h(1-h)} x\Big)v_*(y) + \mathcal{O}( \eps^{\frac 52})\\
                    \label{eqn:leadingOrderSolnuveta:eta}\eta^\eps(x) &= -\frac{3\eps}{\theta} \sech^2 \Big(\frac{\sqrt{3\eps}}{2h(1-h)} x\Big) + \mathcal{O}(\eps^2).    
                \end{align}
            \end{subequations}
        \item These functions satisfy the estimates             
        \begin{subequations}\label{eqn:SizeOfuxvxetax}
            \begin{gather}
            \label{eqn:SizeOfuxvxetax:x}
            \sup_x \lVert u^\eps_x(x, \, \cdot \, ) \rVert_{H^1} = \mathcal{O}(\eps^{\frac 32}), \quad
            \sup_x \lVert v^\eps_x(x, \, \cdot \, ) \rVert_{H^1} = \mathcal{O}(\eps^2), \quad
            \sup_x |\eta^\eps_x(x)| = \mathcal{O}(\eps^{\frac 32})\\
            \label{eqn:SizeOfuyvy}
            \sup_{x,y}|u^\eps_y(x,y)|=\mathcal{O}(\eps), \quad
            \sup_{x,y}|v^\eps_y(x,y)|=\mathcal{O}(\eps).
            \end{gather}
         \end{subequations}
        \item These solutions are homoclinic to 0 in $H^1$, i.e., for any fixed $\eps$, we have \begin{equation}\label{eqn:uvetahomoclinic}
            \lim_{x \to \infty} \lVert (u^\eps(x, \, \cdot \, ),v^\eps(x, \, \cdot \, ),\eta^\eps(x))\rVert_{H^1}=\lim_{x \to -\infty} \lVert (u^\eps(x, \, \cdot \, ),v^\eps(x, \, \cdot \, ),\eta^\eps(x))\rVert_{H^1}=0.
        \end{equation}
    \end{enumerate}
\begin{proof}
Applying the inverse of the scaling from \eqref{eqn:RescaledCoords}, and appealing to Proposition~\ref{thm:aepsbepsExist}\ref{thm:atildebtildeLipschitz} yields
\begin{align*}
    a^\eps(x) &= -\frac{3\eps}{\theta} \sech^2 \Big(\frac{\sqrt{3\eps}}{2h(1-h)} x\Big) + \mathcal{O}(\eps^2)\\
    b^\eps(x) &= \frac{3 \sqrt 3 \eps^{\frac{3}{2}}}{h(1-h)\theta} \tanh \Big(\frac{\sqrt{3\eps}}{2h(1-h)} x\Big) \sech^2 \Big(\frac{\sqrt{3\eps}}{2h(1-h)} x\Big) + \mathcal{O}( \eps^{\frac 52}),
\end{align*}
where the $\mathcal O$ is with respect to the $C^1$ norm. We now consider \eqref{eqn:formOfSolnOnCMTbig}. 
This, the parity of the components of $\psi$ with respect to $b$, and the parity of $a^\eps$ and $b^\eps(x)$ gives us the required parity on $u^\eps$, $v^\eps$ and $\eta^\eps$.  Using \eqref{eqn:orderOfpsi}, and the fact that the oddness of $e_2 \cdot \psi$ with respect to $b$ gives us that $e_2 \cdot \psi= b\mathcal{O}(|a|+|b|+|\eps|)$, we see that
\begin{align*}
    u^\eps(x,y) &= -\frac{3\eps}{\theta} \sech^2 \Big(\frac{\sqrt{3\eps}}{2h(1-h)} x\Big)u_*(y) + \mathcal{O}(\eps^2)\\
    v^\eps(x,y) &= \frac{3 \sqrt 3 \eps^{\frac{3}{2}}}{h(1-h)\theta} \tanh \Big(\frac{\sqrt{3\eps}}{2h(1-h)} x\Big) \sech^2 \Big(\frac{\sqrt{3\eps}}{2h(1-h)} x\Big)v_*(y) + \mathcal{O}( \eps^{\frac 52})\\
    \eta^\eps(x) &= -\frac{3\eps}{\theta} \sech^2 \Big(\frac{\sqrt{3\eps}}{2h(1-h)} x\Big) + \mathcal{O}(\eps^2),
\end{align*}
where the $\mathcal O$s are with respect to the $C^0(\R, H^1((0,1)))$ norm. Differentiating \eqref{eqn:formOfSolnOnCMTbig} with respect to $x$ gives
\begin{align}\nonumber 
    \begin{pmatrix}
      u_x(x,y)\\
      v_x(x,y)\\
      \eta_x(x)
    \end{pmatrix}
   & = a^\eps_x(x) \begin{pmatrix}
      u_*(y)\\
      0\\
      1
    \end{pmatrix}
    +  b^\eps_x(x) \begin{pmatrix}
      0\\
      v_*(y)\\
      0
    \end{pmatrix}
    + \ a^\eps_x(x) \psi_a (a^\eps(x),b^\eps(x),\eps)(y) \\ &
    \qquad +   b^\eps_x(x)\psi_b (a^\eps(x),b^\eps(x),\eps)(y). \label{eqn:formOfSolnOnCMTbigDiffed}
\end{align}
We deduce \eqref{eqn:SizeOfuxvxetax:x} from this, and the fact $D \psi=\mathcal O (|a|+|b|+|\eps|)$. For the other two estimates, notice that we can solve \eqref{eqn:stream2:divfree2} and \eqref{eqn:stream2:lap2} for $u_y$ and $v_y$, as the determinant of the linear system is nonzero so long as $\eta,\eta_x$ are sufficiently small. 
Doing so shows that  $(u^\eps_y(x,y),v^\eps_y(x,y))$ is an analytic function of $\eps$, $u^\eps(x,y)$, $u^\eps(x,h)$, $u^\eps_x(x,y)$, $v^\eps(x,y)$, $v^\eps(x,h)$, $v^\eps_x(x,y)$, $\eta^\eps(x)$, and $\eta^\eps_x(x)$. Calling this function $f\colon \R^9 \to \R^2$, we know all the arguments of $f$ are $\mathcal{O}(\eps)$ uniformly in $x$ and $y$. 
Since $f(0)=0$,
we see we have the estimates \eqref{eqn:SizeOfuyvy}.

Finally, these solutions are homoclinic to 0 because $(\tilde a^\eps, \tilde b^\eps)$ is homoclinic to 0, thus $(a^\eps,b^\eps)$ is homoclinic to 0. Since $\psi(0,0,\eps)=0$, we have $(u^\eps,v^\eps,\eta^\eps)$ is homoclinic to 0.
\end{proof}
\end{prop}

We would like to convert Proposition~\ref{thm:uvetaapprox} into statements about $U^\eps$, $V^\eps$, and $\eta^\eps$. Before we do this, however, it is useful to determine their regularity.
Viewing $(u,v,\eta)$ as a function of $x$, the centre manifold theorem tells us $(u^\eps,v^\eps,\eta^\eps) \in C^N(\R,\mathcal{X})$ for $\eps$ sufficiently small. This gives us immediately that $\eta \in C^N(\R)$.
Showing $U^\eps$ and $V^\eps$ are analytic is a little more involved. Firstly we show that $u^\eps$ and $v^\eps$ are in $H^1$. We can then use the smoothness of the coordinate change \eqref{eqn:coords} to show that $U^\eps$ and $V^\eps$ are in $H^1$. This is enough regularity on the functions to mean we can apply elliptic regularity results, in particular, by examining \eqref{eqn:stream}, we see $U^\eps$ and $V^\eps$ are harmonic, and thus real analytic, away from the interface.

We begin with the following lemma.
\begin{lem}\label{thm:uvareH1}
For any $M>0$, we have $u^\eps,v^\eps \in H^1((-M,M)\times (0,1))$.
\begin{proof}
We know that for each $x$, the function $y \mapsto u^\eps(x,y)$ has a weak derivative in $L^2((0,1))$, which we call $u_y(x,y)$. We also know from the centre manifold theorem that $\lVert u(x, \, \cdot \, ) \rVert_{H^1((0,1))}$ varies continuously with $x$.
We have that $u^\eps_y( \, \cdot \, , \, \cdot \, ) \in L^2((-M,M)\times (0,1))$, as
\begin{align*}
    \int_{-M}^{M} \int_0^1 u^\eps_y(x,y)^2 \; dy \; dx &\leq \int_{-M}^{M} \lVert u^\eps(x, \, \cdot \, ) \rVert^2_{H^1((0,1))} \; dx \\
    &\leq 2M \sup_{x \in [-M,M]}\lVert u^\eps(x, \, \cdot \, ) \rVert^2_{H^1((0,1))} 
    < \infty
\end{align*}
by continuity. A straightforward application of Fubini's theorem then shows that $u^\eps_y$ is the weak derivative of $u^\eps$ in $H^1((-M,M)\times (0,1))$.

We now find the $x$ derivative. The fact that $u^\eps \in C^1(\R, H^1((0,1))$ means that for each $x$, there exists a function of $y$ in $H^1((0,1))$, which we call $u^\eps_x(x,y)$, such that $\lVert u^\eps_x(x, \, \cdot \, ) \rVert_{H^1((0,1))}$ varies continuously with $x$, and that
\begin{align*}
    \lim_{h \to 0}\frac{\lVert u^\eps(x+h, \, \cdot \, )-u^\eps(x, \, \cdot \, )-u^\eps_x(x, \, \cdot \, )h \rVert_{H^1((0,1))}}{|h|} = 0.
\end{align*}
Therefore, since the trace map is bounded on $H^1((0,1))$, we have that for all $y$, 
\begin{align*}
    \frac{\lvert u^\eps(x+h,y)-u^\eps(x,y)-u^\eps_x(x,y)h \rvert}{|h|} \to 0,
\end{align*}
or in other words, $u^\eps_x(x,y)$ is the classical partial derivative of $u^\eps(x,y)$, so it is certainly the weak derivative as well.  We have that $u^\eps_x( \, \cdot \, , \, \cdot \, ) \in L^2((-M,M)\times (0,1))$, as
\begin{align}
\begin{split}\label{eqn:uxIsL2}
    \int_{-M}^{M} \int_0^1 u^\eps_x(x,y)^2 \; dy \; dx &= \int_{-M}^{M} \lVert u^\eps_x(x, \, \cdot \, ) \rVert^2_{L^2((0,1))} \; dx \\
    &\leq \int_{-M}^{M} \lVert u^\eps_x(x, \, \cdot \, ) \rVert^2_{H^1((0,1))} \; dx \\
    &\leq 2M \sup_{x \in [-M,M]}\lVert u^\eps_x(x, \, \cdot \, ) \rVert^2_{H^1((0,1))} < \infty
\end{split}
\end{align}
by continuity. Rewriting \eqref{eqn:uxIsL2} with $u^\eps_x$ replaced by $u^\eps$ shows that $u^\eps$ is in $ L^2((-M,M)\times (0,1))$ as well.
Therefore, since first order weak derivatives of $u^\eps$ exist, and they and $u^\eps$ are in  $ L^2((-M,M)\times (0,1))$, we have that  $u^\eps \in  H^1((-M,M)\times (0,1))$ as required.

The case for $v^\eps$ follows similarly.
\end{proof}
\end{lem}

We now use this to show a similar result for $U^\eps$ and $V^\eps$.
\begin{lem}\label{thm:UVareH1}
    For all $M>0$, we have $U^\eps,V^\eps \in  H^1((-M,M)\times (0,1))$.
\begin{proof}
Using Lemma~\ref{thm:uvareH1}, it is straightforward to show that \[(u^\eps+\omega(y-h)+c)Y_y^{-1} \in  H^1((-M,M)\times (0,1)).\]
Now in order to show that $U^\eps$ and $V^\eps$ are in $H^1((-M,M)\times (0,1))$, all we need to show is that composition with the change of coordinates gives a bounded linear map from $H^1((-M,M)\times (0,1))$ to itself.
Recall the change of coordinates \eqref{eqn:coords:x}--\eqref{eqn:coords:y}.
Given $\mathfrak{u} \in H^1 \cap C^\infty$, we let \[U(X,Y)=\mathfrak{u}(x(X,Y),y(X,Y)).\] 
It can be seen that that 
\begin{equation}
\begin{aligned}
    \lVert U \rVert_{H^1}^2 &\leq \frac{2 \sup (1+y_X^2+y_Y^2)}{\inf |y_Y|} \lVert \mathfrak u \rVert_{H^1}^2\\&
    =\frac{2 (1+\mathcal{O}(\eps^3)+1+\mathcal{O}(\eps))}{1-\mathcal{O}(\eps)} \lVert \mathfrak u \rVert_{H^1}^2
    \leq 5 \lVert \mathfrak u \rVert_{H^1}^2,
\end{aligned}    
\end{equation}
where the second equality uses the formula for $y$, \eqref{eqn:coords:y}, and the estimates on $\eta^\eps$ in \eqref{eqn:leadingOrderSolnuveta:eta} and \eqref{eqn:SizeOfuxvxetax:x}.
We deduce the change of coordinates is a bounded linear map on $ H^1 \cap C^\infty$, but this is a dense subspace of $H^1$, so we have our required result.
\end{proof}
\end{lem}

We are now ready to apply elliptic regularity to show $U^\eps$ and $V^\eps$ are analytic.

\begin{prop}
    The velocity components $U^\eps$ and $V^\eps$ are analytic on $\Omega_0 \cup \Omega_1$.
\begin{proof}

Notice that $U^\eps$ and $V^\eps$ are weakly harmonic in $\Omega_0$ and in $\Omega_1$, as for all $\phi \in C^\infty_c(\Omega_i)$, we have
\begin{align*}
    \int_{\Omega_i} \nabla U^\eps  \cdot \nabla \phi \; dX \; dY &= \int_{\Omega_i} U^\eps_X \phi_X + U^\eps_Y \phi_Y \; dX \; dY\\
     &= \int_{\Omega_i} -V^\eps_Y \phi_X + (\omega +V^\eps_X) \phi_Y \; dX \; dY\\
     &= \int_{\Omega_i} V^\eps \phi_{XY} - V^\eps \phi_{XY} \; dX \; dY=0,
\end{align*}  
and similarly for $V^\eps$. By standard elliptic regularity arguments, $U^\eps$ and $V^\eps$ are therefore smooth harmonic functions, and hence real-analytic.
\end{proof}
\end{prop}

We are now ready to prove results on $U^\eps$ and $V^\eps$ of a similar form to those in Proposition~\ref{thm:uvetaapprox}.
\begin{cor}\label{thm:cors about UVeta}
    \hfill 
    \begin{enumerate}[label=\rm(\alph*)]
        \item \label{thm:cors about UVeta:parity}  $U^\eps,\eta^\eps$ are even in $x$, and $V^\eps$ is odd in $x$.
        \item\label{thm:cors about UVeta:homoclinic}  $(U^\eps,V^\eps,\eta^\eps)$ is homoclinic to $(\omega(Y-h)+h(1-h)+\eps,0,0)$ with respect to the $C^0$ norm. 
        \item\label{thm:cors about UVeta:small} $U^\eps_X $ and $ \eta^\eps_X$ are $\mathcal{O}(\eps^{\frac 32})$, and $V^\eps_X$ is $\mathcal{O}(\eps^{2})$ with respect to the $C^0$ norm.
    \end{enumerate}    
    \begin{proof}
        Conclusion~\ref{thm:cors about UVeta:parity} is immediate. Conclusion~\ref{thm:cors about UVeta:homoclinic} follows by arguing similarly to in Lemma~\ref{thm:UVareH1}. We see that the coordinate change \eqref{eqn:coords} gives a bounded linear map from $H^1((0,1))$ to itself. Therefore \eqref{eqn:uvetahomoclinic} implies     
        \begin{align*}
        \lim_{X\to \pm \infty} \lVert U^\eps(X, \, \cdot \, )-\omega( \, \cdot \, -h)-h(1-h)-\eps\rVert_{H^1} =0,
        \qquad
        \lim_{X\to \pm \infty} \lVert V^\eps(X, \, \cdot \, )\rVert_{H^1} =0,
        \end{align*}
        as required, and we already know $\eta^\eps(X)$ is homoclinic to 0.

        For conclusion~\ref{thm:cors about UVeta:small}, the estimate on $\eta^\eps_X$ comes immediately from \eqref{eqn:SizeOfuxvxetax:x}. As for the other two, again, we argue similarly to in Lemma~\ref{thm:UVareH1}, and consider \eqref{eqn:UVderivatives} to see 
        \begin{align*}
        U^\eps_X&=Y_{XY}\mathcal{O}(1) + y_Y u^\eps_x 
        =(\eta^\eps_X+u^\eps_x)\mathcal{O}(1)
        =\mathcal{O}(\eps^{\frac 32})\\
        V^\eps_X&=v_x + Y_X v_y=\mathcal{O}(\eps^{2}),
        \end{align*}
        as required.
    \end{proof}
\end{cor}

We now have everything we need to show that $\overline{U}^\eps,\overline{V}^\eps,\overline{\eta}^\eps$ satisfy the estimates in Theorem~\ref{thm:main}.
\begin{thm}\label{thm:UVetaapprox}
    The estimates on $\overline{U}^\eps,\overline{V}^\eps,\overline{\eta}^\eps$ in Theorem~\ref{thm:main} are true.
\begin{proof}
    Recall \eqref{eqn:coords}, and the definition of $\overline{\eta}^\eps$ from Theorem~\ref{thm:main}. It will be useful here to define the quantities
    \begin{align*}
        H&= \begin{cases}
            -h^{-1} & \quad \text{for} \quad 0 \leq y \leq h\\
            (1-h)^{-1} & \quad \text{for} \quad h \leq y \leq 1
            \end{cases}\\
        h^+&=\sup_{X \in \R} y(X,h+\overline{\eta}^\eps(X))\\
        h^-&=\inf_{X \in \R} y(X,h+\overline{\eta}^\eps(X)).
    \end{align*}
    Notice that $h^+=h+\mathcal{O}(\eps^2)$, and similarly for $h^-$.
    We have just shown the required result for $\eta$ in Conclusion~\ref{thm:uvetaapprox:leadingTerms} of Proposition~\ref{thm:uvetaapprox}. We deal with $U$ first.
    In what follows, $X$ and $Y$ should be understood to mean $X(x,y)$ and $Y(x,y)$ respectively. Using \eqref{eqn:coords:u} and \eqref{eqn:leadingOrderSolnuveta}, we see that
    \begin{align}
    U^\eps(X,Y) &= y_Y( \omega(y-h) + h(1-h)) +\eps +\eta^\eps(x)u_*(y) + \mathcal{O}(\eps^2) \nonumber
    \\
    \label{eqn:an equation with H inside}
    &=  \omega(y-h) + h(1-h) +\eps  + \frac{\eta^\eps(x)}{h(1-h)} \left(\omega y(1-y)+h^2(1-h)^2H \right) + \mathcal{O}(\eps^2) .
    \end{align}  
    Notice that despite the discontinuities in $\omega$ and $H$, we have that $\omega y(1-y)+h^2(1-h)^2H$ is continuous at $y=h$, and it is in fact Lipschitz continuous. 
    Therefore, for $y$ between $0$ and $h^+$, replacing $\omega$ with $\omega_0$ and $H$ with $-1/h$ in \eqref{eqn:an equation with H inside} introduces errors of $\mathcal{O}(\eps^2)$. More precisely, for $(x,y) \in \R \times [0,h^+]$, we have that
    \begin{align*}
    U^\eps(X,Y)&=  \omega_0(y-h) + h(1-h) +\eps  + \frac{\eta^\eps(x)}{h(1-h)} \left(\omega_0 Y(1-Y)-h(1-h)^2 \right) + \mathcal{O}(\eps^2),
    \end{align*}
    where as usual, the error term is uniform in $(x,y)$. 
    Notice that if $(X,Y)$ satisfies $Y \leq h+\overline{\eta}^\eps(X)$,  then it also satisfies $Y \leq Y(x,h^+)$. This means that when we apply the coordinate transformation \eqref{eqn:coords:y}, we see that for all $X,Y$ satisfying $Y \leq  h+\overline{\eta}^\eps(X)$, we have 
    \begin{align*}
    U^\eps(X,Y)&=  \omega_0(Y-h) + h(1-h) +\eps  -(1-h)\eta^\eps(x)+ \mathcal{O}(\eps^2) \\
    &=  \overline{U}^\eps(X,Y)+ \mathcal{O}(\eps^2).
    \end{align*}
    We argue similarly to conclude that for all $X,Y$ satisfying $Y \geq  h+\overline{\eta}^\eps(X)$, we have 
    \begin{align*}
    U^\eps(X,Y)&=\omega_1(Y-h) + h(1-h) +\eps  +h\overline{\eta}^\eps(x)+ \mathcal{O}(\eps^2) 
    =  \overline{U}^\eps(X,Y) + \mathcal{O}(\eps^2).
    \end{align*} 
    Showing the result for $V$ follows a similar but more straightforward argument. We see for $(X,Y)$ with $Y\leq h+\overline{\eta}^\eps(X)$, we have
    \begin{align*}
        V^\eps(X,Y)&= \frac{3 \sqrt 3 \eps^{\frac{3}{2}}}{\theta} \tanh \Big(\frac{\sqrt{3\eps}}{2h(1-h)} x\Big) \sech^2 \Big(\frac{\sqrt{3\eps}}{2h(1-h)} x\Big)(1-H(y-h)) + \mathcal{O}( \eps^{\frac 52})\\
        &= -\frac{3 \sqrt 3 \eps^{\frac{3}{2}}}{\theta h} \tanh \Big(\frac{\sqrt{3\eps}}{2h(1-h)} x\Big) \sech^2 \Big(\frac{\sqrt{3\eps}}{2h(1-h)} x\Big)y + \mathcal{O}( \eps^{\frac 52})\\
        &= -\frac{3 \sqrt 3 \eps^{\frac{3}{2}}}{\theta h} \tanh \Big(\frac{\sqrt{3\eps}}{2h(1-h)} X\Big) \sech^2 \Big(\frac{\sqrt{3\eps}}{2h(1-h)} X\Big)Y + \mathcal{O}( \eps^{\frac 52})\\
        &= \overline{V}^\eps(X,Y)+ \mathcal{O}( \eps^{\frac 52}),
    \end{align*}
    and similarly, for $(X,Y)$ with $Y\geq h+\overline{\eta}^\eps(X)$, we have
    \begin{align*}
        V^\eps(X,Y)&= \frac{3 \sqrt 3 \eps^{\frac{3}{2}}}{\theta (1-h)} \tanh \Big(\frac{\sqrt{3\eps}}{2h(1-h)} X\Big) \sech^2 \Big(\frac{\sqrt{3\eps}}{2h(1-h)} X\Big)(1-Y) + \mathcal{O}( \eps^{\frac 52})\\
        &= \overline{V}^\eps(X,Y)+ \mathcal{O}( \eps^{\frac 52}).
        \qedhere
    \end{align*}
\end{proof}
\end{thm}

In a sense, \eqref{eqn:leadingOrderSolnuveta} is a more precise approximation of the solution than $\overline{U}^\eps, \overline{V}^\eps, \overline{\eta}^\eps$. This is because in \eqref{eqn:leadingOrderSolnuveta}, the discontinuity in the derivatives happens at $y=h$, which is where it occurs in the exact solution, but the discontinuity in the derivatives of $\overline{U}^\eps$, $\overline{V}^\eps$, happens at $Y=h+\ \overline{\eta}^\eps$. 

\section{Streamline patterns}\label{sec:SignOfEta}

In this section we will prove Theorem~\ref{thm:main2}. Both the qualitative nature of the results we show, and the general method used are similar in spirit to those in, for example \cite{Wahlen:critLayer}.
From now onward we assume without loss of generality that $\omega_0 \leq 1-h$.
We lose no generality from this, as if it is not the case, then performing the vertical reflection
\begin{equation}\label{eqn:reflecting in Y}
    y \mapsto 1-y, \quad h \mapsto 1-h, \quad V \mapsto -V
\end{equation}
and relabelling $\Omega_0$, $\Omega_1$, puts us in the regime where once again $\omega_0 \leq 1-h$ holds. After performing this reflection, we still have the vorticity of lower layer is one greater than the vorticity of the upper layer.

\subsection{Signs of components of the flow}\label{sec:signsofcomponents}

One thing we can deduce quite straightforwardly is a sign for $U^\eps$ on the interface, as
\begin{align}\label{eqn:signUatinterface}
    U^\eps(X,h+\eta^\eps(X))&=h(1-h)+\mathcal O (\eps)
    >0.
\end{align}
We will next show that the interface is monotone, that is that $\eta^\eps_x(x)$ has one sign for $x>0$ and the opposite sign for $x<0$.
\begin{prop}\label{thm:SignOfetax}
    For all $x$, $\eta^\eps_x(x)$ has the same sign as $\theta x$.
\end{prop}
\begin{proof}
At $x=0$ we have that $a^\eps_x(0)=\eta^\eps_x(0)=0$ by evenness of $a^\eps$ and $\eta^\eps$.  
We now show the $x \neq 0 $ case. A corollary of Proposition~\ref{thm:aepsbepsExist} is that for all $x$, $\tilde b^\eps(x)$ shares a sign with $x$. Therefore, for $x \neq 0$,
\begin{equation}\label{eqn:signOnb}
    \frac{b^\eps(x)}{\theta x}>0.
\end{equation}
In particular, by Remark~\ref{thm:uniqe0oftildeb}, $b^\eps(x)=0$ if and only if $x=0$.
We also know using \eqref{eqn:axandbx:ax}, that for $x \neq 0$, 
\begin{equation}\label{eqn:axandbagree}
    \frac{a^\eps_x(x)}{b^\eps(x)}= 1+\mathcal O (|a^\eps(x)|+|b^\eps(x)|+|\eps|)>0.
\end{equation}
Therefore, when $x \neq 0$, we know that $a^\eps_x(x) \neq 0$.

Consider \eqref{eqn:formOfSolnOnCMTbigDiffed}. We are interested in $\eta^\eps_x$ so we take the third component of each vector, then divide the resulting equation by $a^\eps_x$. 
Since $\psi_a$ and $\psi_b$ are 3-component vectors, we continue our notational convention, and denote their third components by $e_3 \cdot \psi_a$ and $e_3 \cdot \psi_b$ respectively. 
We also know from \eqref{eqn:reversibilityOfPsi} that $e_3 \cdot \psi$ is even in $b$ and $C^N$, hence $e_3 \cdot \psi_b$ is odd in $b$ and $C^{N-1}$, therefore 
\[e_3 \cdot \psi_b(a^\eps, b^\eps, \eps )= b^\eps(1+\mathcal O (|a^\eps|+|b^\eps|+|\eps|))= b^\eps(1+\mathcal O (\eps)).\]
Putting all this together, we get
\begin{align}\label{eqn:compareEtaxAndax}
    \frac{\eta^\eps_x(x)}{a^\eps_x(x)}&= 1 +  e_3 \cdot \psi_a(a^\eps,b^\eps,\eps) + \frac{b^\eps_x(x)b^\eps(x)}{a^\eps_x(x)} (1+\mathcal O (\eps)).
\end{align}
We know from differentiability of $\psi_a$ that $e_3 \cdot \psi_a=\mathcal O (|a^\eps|+|b^\eps|+|\eps|)=\mathcal O (\eps)$, hence, multiplying \eqref{eqn:compareEtaxAndax} by $a^\eps_x / b^\eps$, we see
\begin{align*}
    \frac{\eta^\eps_x(x)}{b^\eps(x)}&=(1+\mathcal O (\eps))\frac{a^\eps_x(x) }{ b^\eps(x)} + b^\eps_x(x) (1+\mathcal O (\eps))\\
    &=1+\mathcal O (\eps) + b^\eps_x(x) \big(1+\mathcal O (|a^\eps(x)|+|b^\eps(x)|+|\eps|\big))
    =1+\mathcal O (\eps)
    >0,
\end{align*}
where we have used the fact that $ b^\eps_x(x) = \mathcal{O}(\eps^2)$, which can be deduced from the form of $F$ in \eqref{eqn:directODE}, the rescaling \eqref{eqn:RescaledCoords} and the boundedness of $\tilde a^\eps$ and $\tilde b^\eps$. Therefore, we see that $\eta^\eps_x$, $b^\eps$, and $\theta x$ share a sign.
\end{proof}
\begin{cor}\label{thm:SignOfeta}
    For all $x$, we have $\theta \eta^\eps(x)<0$.
    \begin{proof}
        We know by conclusion~\ref{thm:cors about UVeta:homoclinic} of Corollary~ \ref{thm:cors about UVeta} that $\eta^\eps$ tends to $0$ as $x$ tends to $\pm \infty$. This and Proposition~\ref{thm:SignOfetax} give us the result.
    \end{proof}
\end{cor}

We now show that $V^\eps$ also shares a sign with $\theta X$, except on the upper and lower boundaries, where we know $V^\eps=0$. We use a maximum principle argument, one possible reference for which would be \cite[Section 6.4, Theorem 3]{Evans:pde}.

\begin{prop}\label{thm:signOnV}
    For all $(X,Y) \in \R \times (0,1)$, we have that $V^\eps(X,Y)$ and $\theta X$ share a sign.
\begin{proof}
Notice that by oddness, we have immediately that $V^\eps(0,Y)=0$.
We now consider the case where $X>0$ and $\theta>0$, but the other cases follow an almost identical argument.
Let $\Omega_i^{M}= \{ (X,Y) \in \Omega_i \mid 0<X<M\}$. Since $V^\eps$ is harmonic on $\Omega_i^{M}$, we can apply the strong maximum principle, which says $V^\eps$ attains a minimum on $\overline{\Omega_i^{M}}$, and does so only on $\partial \Omega_i^{M}$.
We know that $V^\eps(X,0)=0$ from \eqref{eqn:stream:kintop}, and that $V^\eps(0,Y)=0$ by oddness. We also know from \eqref{eqn:stream:kinint} that
\begin{align*}
    V^\eps(X,h+\eta^\eps(X))=\eta^\eps_X(X)U^\eps(X,h+\eta^\eps(X)),
\end{align*}
therefore,
\[  V^\eps(X,h+\eta^\eps(X))) >0. \]
We know from Corollary~\ref{thm:cors about UVeta} that $\lVert V^\eps(M,  \, \cdot \, ) \rVert_{H^1((0,1))}$ tends to 0 as $M$ tends to infinity. This means we can use an argument by contradiction to show that $V^\eps>0$ on $\Omega_i^\infty$.

Suppose there exists $(\hat{X},\hat{Y})$ with $\hat{X}>0$ such that $V^\eps(\hat{X},\hat{Y})<0$. Since 
\[\lVert V^\eps(X,  \, \cdot \, ) \rVert_{H^1((0,1))} \to 0,\]
there exists $M>\hat{X}$  such that for all $Y \in [0,1]$, we have \[|V^\eps(M,Y)|<\tfrac 12 |V^\eps(\hat{X},\hat{Y})|.\] 
Hence, the minimum $V^\eps$ attains on $\partial \Omega_i^{M}$ is no less than $\frac 12 V^\eps(\hat{X},\hat{Y})$, and so $V^\eps$ does not attain its minimum on the boundary, contradicting the maximum principle. 
Thus, for all $M>0$, we have that $V^\eps \geq 0$ on $\overline{\Omega_i^{M}}$. 
However, we know that $V^\eps$ does not attain its minimum in $\Omega_i^{M}$, therefore $V^\eps>0$ on $\Omega_i^{M}$. Therefore, for all $X>0,Y\in(0,1)$, we have $V^\eps(X,Y)>0$.

The cases where $\theta<0$ or $X<0$ follow very similarly, and we conclude that for all $(X,Y)\in \R \times (0,1)$, we have that $V^\eps(X,Y)$ and $\theta X$ share a sign.
\end{proof}
\end{prop}

We are now in a position to start describing the critical layers and stagnation points of the flow. Recall the critical layer of the solution is the set on which $U^\eps=0$.
We start by identifying the stagnation points in the trivial solutions. If $\omega_0=1-h$, then the trivial solution stagnates at $y=0$ and at $y=1$. 
Otherwise the trivial solution stagnates in the upper layer at $y=h-h(1-h)/\omega_1$ , since recall we are now assuming $\omega_0 \leq 1-h$.

We now investigate the stagnation points of the non-trivial solutions. We separate into two cases, which occupy different regions of parameter space. It  turns out that the critical layers are qualitatively very different; in the first case they are unbounded, and in the second, bounded.

\subsection{Unbounded critical layer}\label{sec:Unbddstagnation}

We first assume that $1-h > \omega_0$, i.e., that stagnation occurs in $\Omega_1$, rather than on the upper or lower boundaries. 
Notice that if we also have $1-2h>\omega_0$, then we are outside the $\omega = \gamma(\Psi)$ regime.

We now show there is a critical layer, and investigate some of its properties.
\begin{lem}
    Suppose $1-h>\omega_0$. Given $\eps$, there exists a unique function $Y_*(X)$ such that $U^\eps(X,Y_*(X))=0$. This function $Y_*$ is analytic.
\begin{proof}
First notice that in $\Omega_0$, we have
\begin{align*}
    U^\eps(X,Y)&=\omega_0(Y-h)+h(1-h)+\mathcal{O}(\eps)
    >h\min(1-h,1-h-\omega_0)+\mathcal{O}(\eps)
    >0,
\end{align*}
so any zeros of $U^\eps$ must be in $\Omega_1$. 
Notice for all $X$, we have 
\begin{align*}
    U^\eps(X,1)&=\omega_1(1-h)+\mathcal{O}(\eps)<0,
\end{align*}
and \eqref{eqn:signUatinterface} tell us that
\[ U^\eps(X,h+\eta^\eps(X))>0.\]
Therefore, by the intermediate value theorem, for all $X$, there exists $Y_* \in (h+\eta^\eps(X),1)$ such that $U^\eps(X,Y_*)=0$.

We now show this $Y_*$ is unique. By \eqref{eqn:stream:lap}, in $\Omega_1$, 
\begin{align*}
    U^\eps_Y&=\omega_1+V^\eps_X
    \omega_1+\mathcal{O}(\eps^2)
    <0.
\end{align*}
This implies that for each $X$, $U^\eps$ is a strictly decreasing function of $Y$, so in fact, for every $X$, there exists a unique $Y_*(X) \in (h+\eta^\eps(X),1)$ such that $U^\eps(X,Y_*)=0$. Since $U^\eps$ is an analytic function, and $U^\eps_Y(X,Y_*(X)) \neq 0$, we can apply the analytic implicit function theorem to see that $Y_*$ is an analytic function of $X$. 
\end{proof}
\end{lem}

\begin{rk}
    Solving \eqref{eqn:coords:u}, we see that if  $U^\eps(X,Y)=0$ for some $(X,Y) \in \Omega_1$, then $y^\eps(X,Y)=h-h(1-h)/\omega_1+\mathcal{O}(\eps)$. 
    By smoothness of the coordinate transformations, we see that $Y=h-h(1-h)/\omega_1+\mathcal{O}(\eps)$, or in other words, the distance between the critical layer of our solution, and of the background shear flow is of order $\eps$, uniformly in $X$. In particular, for $\omega_0<1-h$, the critical layer $Y_*$ does not touch the boundary $Y=1$ or the interface.
\end{rk}

We can deduce from this lemma that we have a stagnation point at $(0,Y_*(0))$, and will now examine its behaviour.

\begin{thm}\label{thm:unbddStagnation}
    Suppose $1-h>\omega_0$. The flow given by $(U^\eps,V^\eps)$ has a unique stagnation point at $(0,Y_*(0))$, which is a centre if $\theta>0$, and a saddle if $\theta <0$.
\begin{proof}
We know that $U^\eps(0,Y)= 0$ if and only if $Y = Y_*(0)$, and that for $Y\neq 0,1$, the vertical velocity $V^\eps(X,Y)=0$ if and only if $X = 0$ by Propostition~\ref{thm:signOnV}. Therefore
we have a unique stagnation point at $(0,Y_*(0))$. We now consider the nature of this stagnation point.
Let $s,s'$ be the eigenvalues of the derivative of $(U^\eps,V^\eps)$ at a stagnation point. 
We know from incompressibility, and the fact we have real solutions that $s+s'=0$, and $ss'\in \R$. 
Hence, if the determinant of the derivative of $(U^\eps,V^\eps)$ is positive, $s$ and $s'$ are both purely imaginary, and we have a centre.
If the determinant is negative, they are both real, and we have a saddle. 
Notice that by evenness of $U^\eps$,  we have $U^\eps_X(0,Y) =0$ for all $Y$.
Using this and \eqref{eqn:stream:lap}, we have
\begin{align*}
    \det(D(U^\eps,V^\eps)(0,Y)) &= U^\eps_X V^\eps_Y - U^\eps_Y V^\eps_X
    =-V^\eps_X(\omega_1 + V^\eps_X).
\end{align*}

We know that in $(0,\infty)\times(0,1)$, we have $\theta V^\eps>0$, that $V^\eps(0,Y)=0$, and that $V^\eps$ is harmonic, so we can apply the Hopf lemma to deduce that for all $Y \in (0,1)$, we have $\theta V^\eps_X(0,Y)>0$. We also see from Conclusion~\ref{thm:cors about UVeta:small} of Corollary~\ref{thm:cors about UVeta} that
\[\omega_1 + V^\eps_X=\omega_1+\mathcal O (\eps^2)<0,\]
so therefore
\begin{align*}
    \theta \det(D(U,V)(0,Y)))>0,
\end{align*} 
and in particular, this is true at $Y=Y_*(0)$, the stagnation point.
In other words, if $\theta>0$, we have a centre, if $\theta<0$, we have a saddle.
\end{proof}
\end{thm}

\subsection{Bounded critical layer}\label{sec:bddstagnation}

We now consider the case $\omega_0=1-h$, $\omega_0-\omega_1=1$. This implies $\theta=2h-1$.
Notice that it is sufficient to consider cases when $h\neq \frac 12$, as $\theta=0 \text{ if and only if } h=\frac 12$.
Notice also, that we can reflect in $Y$, as we did in \eqref{eqn:reflecting in Y}, to insist that $h> \frac 12$. This will imply that $\theta>0$.

We will now show that the critical layer is bounded, and ends where it intersects with the upper boundary.

\begin{thm}\label{thm:bddStagnation}
    Suppose $\omega_0=1-h$, $\omega_1=-h$, $\frac 12 < h < 1$. There exists $X_*>0$ such that the set on which $U^\eps=0$ is given by a the graph of an analytic function $Y_* \colon [-X_*,X_*] \to [0,1]$, satisfying $Y_*(X_*)=Y_*(-X_*)=1$. Let $R$ be the region
    \begin{equation}\label{eqn:define R}
        R=\{(X,Y) \mid |X|<X_*, \ Y_*(X)<Y<1\}.
    \end{equation}
    Then $U^\eps<0$ in $R$, and $U^\eps>0$ outside $\overline{R}$.
    Furthermore, the flow given by $(U^\eps,V^\eps)$ has three stagnation points: a centre located at $(0,Y_*(0))$, and saddle points at $(\pm X_*,1)$.
\begin{proof}

We first see that $U^\eps(0,1)<0$, and $U^\eps(0,0)>0$, by observing that 
\begin{align*}
    U^\eps(0,1)&=y_Y(0,1)(u^\eps(0,1)+\omega_1(1-h)+h(1-h)+\eps)\\
    &=y_Y(0,1)(u^\eps(0,1)+\eps).
\end{align*}
Using \eqref{eqn:yderivatives} and \eqref{eqn:formOfSolnOnCMTbig}, we see that
\begin{align*}
    U^\eps(0,1) &=(1+\mathcal{O}(\eps))(a^\eps(0)u_*(1)+\eps+\mathcal O(\eps^2)),
\end{align*}
and from \eqref{eqn:ustar} and \eqref{eqn:RescaledCoords}, we have
\begin{align*}
    U^\eps(0,1)&=(1+\mathcal{O}(\eps))\left(\frac{2h\eps}{\theta}\tilde{a}^\eps(0)+\eps+\mathcal O(\eps^2)\right)\\    &
    =\frac{-3h\eps}{\theta}+\eps+\mathcal O(\eps^2)
    =\frac{-1-h}{2h-1}\eps+\mathcal O(\eps^2)
    <0.
\end{align*}
A very similar argument shows that $U^\eps(0,0)>0$.

Given \eqref{eqn:coords:u} and the limiting behaviour of $u^\eps$ in Proposition~\ref{thm:uvetaapprox}, we see that if $|X|$ is sufficiently large, we have for all $Y$ that
\begin{equation}\label{eqn:bddCaseUeventuallyPositive}
\begin{aligned}
    U^\eps(X,Y)&=y_Y(\omega(y-h)+h(1-h)+\eps + u^\eps(x,y))
    \geq y_Y(\eps + u^\eps(x,y))
    >0.
\end{aligned}    
\end{equation}
Note, this is not true in the unbounded critical layer case. 
We conclude that by the intermediate value theorem, and evenness of $U^\eps$, there exists $X_*>0$ such that 
\[U^\eps(X_*,1)=U^\eps(-X_*,1)=0.\]
Using  \eqref{eqn:stream:kintop}, Proposition~\ref{thm:signOnV}, and the Hopf lemma, we see that for $X \neq 0$,
\[ X V^\eps_Y(X,1)<0, \]
therefore applying \eqref{eqn:stream:divfree} gives us that $XU^\eps_X(X,1)>0$, so $X_*$ is the only positive solution of $U^\eps(X,1)=0$. We also see that if $|X|>X_*$, we have $U^\eps(X,1)>0$, and if $|X|<X_*$, we have $U^\eps(X,1)<0$.

We now note that for all $X \in (-X_*,X_*)$, we have $U^\eps(X,1)<0<U^\eps(X,h+\eta(X))$, therefore there exists $Y_*(X) \in (h+\eta^\eps(X),1)$ such that
$U^\eps(X,Y_*(X))=0$.
We see from \eqref{eqn:stream:lap} and the fact that $\omega_1<0$ that in $\Omega_1$,
\[U^\eps_Y(X,Y)=\omega_1+V^\eps_X(X,Y)=\omega_1+\mathcal O (\eps^2)<0, \]
and similarly in $\Omega_0$, we have $U^\eps_Y(X,Y)>0$. Therefore, since $U^\eps_Y(X,0)>0$, for each $X$ this $Y_*(X)$ is unique. 
We can apply the analytic implicit function theorem to conclude that $Y_*(X)$ is an analytic function.
Next, we seek to show that $Y_*(X) \to 1$ as $X \to X_*$ and as $X \to -X_*$. 
Notice that
\begin{align*}
    \begin{pmatrix}
      -1\\
      -1
    \end{pmatrix} \cdot \nabla U^\eps(X_*,1) &=-U^\eps_X(X_*,1)-U^\eps_Y(X_*,1)
    =-\omega_1+\mathcal O (\eps^{\frac 32})
    >0.
\end{align*}
Thus, there exists $\delta>0$ such that for all $t \in (0,\delta)$, we have $U^\eps(X_*-t,1)<0<U^\eps(X_*-t,1-t)$, therefore
$1-t<Y_*(X_*-t)<1$, so we are done.

Define the bounded region $R$ as in \eqref{eqn:define R}.
We know that $U^\eps(X,Y_*(X))=0$, and that $U(X,1)<0$ for $|X|<X_*$.
Therefore we can apply the maximum principle to show that in $R$, we have $U^\eps<0$.
We can also show that $U^\eps>0$ outside $\overline{R}$. 
To do this we first recall that we deduced in \eqref{eqn:bddCaseUeventuallyPositive} that there exists $M>X_*$ such that for all $X$ with $|X|\geq M$, and all $Y$, we have $U^\eps(X,Y)>0$.

For $X \in (-M,M)$, we can apply the Hopf lemma to $V^\eps$ at $(X,0)$, and deduce with \eqref{eqn:stream:divfree} that $XU^\eps_X(X,0)>0$. We already showed at the very start of this proof that $U^\eps(0,0)>0$, and so we see that for all $X$, $U^\eps(X,0)>0$. Therefore we apply the maximum principle to $U^\eps$ on $((-M,M) \times (0,1)) \setminus \overline{R}$, and conclude that on this set too, $U^\eps>0$.

Finally, we discuss the stagnation point. We deduce from what we have just discussed that there is a stagnation point at $(0,Y_*(0))$, and that there are no others away from the boundaries, and apply the same arguments as in Theorem~\ref{thm:unbddStagnation} to conclude that since the stagnation happens in $\Omega_1$, and $\theta>0$, we have a centre.
\end{proof}
\end{thm}

We now show that in this region of parameter space, we have a streamline which is attached to the upper layer, as shown in Figure~\ref{fig:solutions in regions of parameter space}.

\begin{prop}\label{thm:attachedStreamline}
    Suppose $\omega_0=1-h$, $\omega_1=-h$, $\frac 12 < h < 1$. Recall $X_*$ from Theorem~\ref{thm:bddStagnation}, the $X$ coordinate of where the critical layer meets the upper boundary. There exists a streamline with endpoints $(-X_*,1)$ and $(X_*,1)$.
\begin{proof}
     We first recall that streamlines are level curves of the stream function $\Psi^\eps$. Pick the stream function such that $\Psi^\eps(-X_*, 1)=0$. 
     For notational convenience, let $\tilde{X} = -U^\eps_{Y}(-X_*, 1)$, $\tilde{Y} = -U^\eps_{X}(-X_*, 1)$. 
     We have that $\tilde{X}, \tilde{Y}>0$.
     Notice that for all sufficiently small $t$,
     \begin{align*}
         \Psi^\eps(-X_*+t \tilde{X}, 1-\tfrac 32 t \tilde{Y})
         &= -\tfrac 32 t^2 \tilde{X} \tilde{Y} U^\eps_{X}(-X_*, 1) + \tfrac 98 t^2 \tilde{Y}^2 U^\eps_{Y}(-X_*, 1) + \mathcal{O}(t^3)\\
         &= \tfrac 32 t^2 \tilde{X} \tilde{Y}^2 - \tfrac 98 t^2 \tilde{X} \tilde{Y}^2 + \mathcal{O}(t^3) >0,
     \end{align*}
     and 
     \begin{align*}
         \Psi^\eps(-X_*+t \tilde{X}, 1-3 t \tilde{Y})
         &= -3 t^2 \tilde{X} \tilde{Y} U^\eps_{X}(X_*, 1) + \tfrac 92 t^2 \tilde{Y}^2 U^\eps_{Y}(X_*, 1) + \mathcal{O}(t^3)\\
         &= 3 t^2 \tilde{X} \tilde{Y}^2 - \tfrac 92 t^2 \tilde{X} \tilde{Y}^2 + \mathcal{O}(t^3) <0.
     \end{align*}
     Therefore, by the intermediate value theorem, for all $t>0$ sufficiently small, there exists $Y_{**} \in (1-3t \tilde{Y}, 1-\tfrac 32 t \tilde{Y}) $ such that $\Psi^\eps(-X_*+t \tilde{X}, Y_{**})=0$.

     We also have that $\Psi^\eps$ is increasing in $Y$, in that for all $s \in ( \tfrac 32 , 3) $, we have 
     \begin{equation}
      \label{eqn:UposNearCritStreamline}
     \begin{aligned}
         \Psi^\eps_Y(-X_*+t \tilde{X}, 1- st \tilde{Y})&=U^\eps(-X_*+t \tilde{X}, 1-st \tilde{Y})\\
         &=t \tilde{X} U^\eps_X(-X_*,1)-st\tilde{Y}  U^\eps_Y(-X_*,1)\\
         &=-t \tilde{X} \tilde{Y}+st\tilde{Y}  \tilde{X}>0.
     \end{aligned}   
     \end{equation}
     Thus, for each $X$, this $Y_{**}$ is unique in $(1-3(X_*+X) \tilde{Y}\tilde{X}^{-1}, 1-\tfrac 32 (X_*+X) \tilde{Y}\tilde{X}^{-1}) $ . We can now apply the implicit function theorem to deduce that we have some small $\delta>0$, and a smooth function $Y_{**}(X)$, defined for $X\in [-X_*,-X_*+ \delta]$, such that $\Psi^\eps(X,Y_{**}(X))=0$, and $Y_{**}(X) \to 1$ as $X \to -X_*$. 
     In other words, $(X,Y_{**}(X))$ gives part of a streamline which touches the upper boundary at $(-X_*,1)$.   
     
    We now seek the rest of the streamline. Consider the motion of a fluid particle starting from $(-X_*+\delta,Y_{**}(-X_*+\delta))$.
    Notice that \eqref{eqn:UposNearCritStreamline} shows that $U^\eps$ is positive along $Y_{**}$, so $Y_{**}$ cannot lie in $R$, the region defined in \eqref{eqn:define R} where $U^\eps<0$, which lies above $Y_*$. Therefore, as $t \to - \infty$, the particle will approach $(-X_*,1)$.
    
    We now show that the particle reaches the line $X=0$, so by the parity of $U^\eps$ and $V^\eps$, as $t \to  \infty$, the particle will approach $(X_*,1)$.
    For $X<0$, we see $U^\eps(X,Y_*(X))=0$, and $V^\eps(X,Y_*(X))>0$. 
    Hence, for $X<0$, fluid particles cannot enter $R$, only leave it. Thus, since the fluid particle started outside $R$, it cannot enter $R$ before crossing the line $X=0$. 
    Therefore, while the fluid particle  is in $(-\infty,0] \times [0,1]$, it cannot travel upwards (using Proposition~\ref{thm:signOnV}), or to the left. Note that the particle cannot approach the centre at $(0,Y_*(0))$, so must remain at least some minimum distance $\mu$ from it at all times. 

    Consider the compact set $K=[-X_*+\delta,0]\times[0,1] \setminus B_\mu((0,Y_*(0)))$. This region has no stagnation points in it, therefore, the fluid flow in this region has some minimum speed. The fluid particle cannot move up or left,  travels with some minimum speed, and cannot pass through the lower boundary, therefore must leave the right hand edge of $K$. It cannot do so by coming within $\mu$ of the centre, therefore it must leave at some point on the line $X=0$. Therefore by the parity of $U^\eps$ and $V^\eps$ we are done.
\end{proof}
\end{prop}

\section{Acknowledgements}
KM received partial support through  The Leverhulme Trust  RPG-2020-107. JS recieved support through EPSRC, EP/T518013/1.

\appendix
\section{Proof of Proposition~\ref{thm:aepsbepsExist}}\label{sec:appendix}
\begin{proof} Recall we have fixed $N \geq 4$, and have that $F \in C^N$. Let the terms of the Taylor polynomial of $e_1 \cdot F$ be given by $\mu_{ijk} a^i b^j \eps ^k$, and of $e_2 \cdot F$  by  $ \lambda_{ijk} a^i b^j \eps ^k$. Throughout the following, we have a family of remainder functions, indexed by $j$, such that $r^i_j(\alpha,\beta,\eps)=\mathcal{O}(|\alpha|^i+|\beta|^i+\eps^i)$, and is even in $\beta$.

We know 
\begin{align*}
    a_x=b(1+\mu_{110}a + \mu_{011}\eps + r^2_1(a,b,\eps)).
\end{align*}
Therefore, by the implicit function theorem and implicit differentiation,
\begin{subequations}\label{eqn:bInTermsOfa}
\begin{align}
    \label{eqn:bInTermsOfa:b}b&=a_x(1-\mu_{110}a - \mu_{011}\eps + r^2_2(a,a_x,\eps))\\
    \label{eqn:bInTermsOfa:bx} b_x &= a_{xx}(1+r^1_3(a,a_x,\eps))-\mu_{110}a_x^2 +r^3_4(a,a_x,\eps).
\end{align}
\end{subequations}
Therefore, since we also know that $b_x = \lambda_{101}a \eps + \lambda_{200} a^2 + r^3_5(a,b,\eps)$, we have
\begin{align*}
    a_{xx}=\lambda_{101}a \eps + \lambda_{200} a^2 + \lambda_{020}b^2+\mu_{110}a_x^2+r^3_6(a,a_x,\eps).
\end{align*}
Applying the scaling, we see
\begin{align*}
    \tilde{a}_{\tilde{x}\tilde{x}}=\tilde{a} + \tilde{a}^2+\eps R(\tilde{a},\tilde{a}_{\tilde{x}},\eps).
\end{align*}
Where $R\in C^{N-1}(\R^3)$, and is even in $\tilde{a}_{\tilde{x}}$.

Now we let $z^\eps=\tilde{a}^\eps - \tilde{a}^0$. 
We see that
\begin{align}\label{eqn:ODEforz}
    z^\eps_{\tilde{x}\tilde{x}}=z^\eps + 2 \tilde{a}^0 z^\eps + (z^\eps)^2 + \eps R(\tilde{a}^0 + z^\eps,\tilde{a}^0_{\tilde{x}}  + z^\eps_{\tilde{x}},\eps).
\end{align}
Let $C^n_\bddeven(\mathcal{U})$ be the set of even functions with domain $\mathcal{U}$, and with finite $C^n$ norm. 
We now use the implicit function theorem to show that for $\eps, \lVert z \rVert_{C^2}$ sufficiently small, we have that for each $\eps$, \eqref{eqn:ODEforz} has a unique solution $z^\eps \in C^2_\bddeven$, such that the map $\eps \mapsto z^\eps$ is in $C^{N-1}(\R,C^2(\R))$.
First, we let
\begin{align*}
    K_1 \colon C^0_\bddeven(\R) &\to C^2_\bddeven(\R),\qquad
    f( \, \cdot \, ) \mapsto \int_{-\infty}^\infty -\frac 12 e^{-| \, \cdot \,  - t|} f(t) \; dt.
\end{align*}
This satisfies $K_1(f''-f)=f$, and is bounded from $ C^0_\bddeven $ to $C^2_\bddeven$. If we also define
\begin{align*}
    r \colon C^2_\bddeven(\R) \times \R &\to C^0_\bddeven(\R),\qquad
    (f,\eps)\mapsto f^2 + \eps R(\tilde{a}^0 + f,\tilde{a}^0_{\tilde{x}}  + f_{\tilde{x}},\eps).
\end{align*}
Notice, $r$ does indeed have codomain consisting of even functions, as $\tilde{a}^0$ is even, and $R$ is even in its second argument.
Thus, \eqref{eqn:ODEforz} can be written equivalently as
\begin{align*}
    z^\eps&=2K_1(\tilde{a}^0 z^\eps) + K_1(r(z^\eps,\eps)).
\end{align*}
Let $K_2z=2K_1(\tilde{a}^0z)$. An easy extension of Arzel\`a--Ascoli shows that $L$ is compact from $C^0_\bddeven(\R)$ to $C^0_\bddeven(\R)$. 
Therefore its spectrum is made only of eigenvalues and possibly $0$, which means in particular, if $\ker(K_2-I)=\{0\}$, then $K_2-I$ is continuously invertible from $C^0_\bddeven(\R)$ to $C^0_\bddeven(\R)$.
If $(K_2-I)z=0$, we see $2 \tilde{a}^0 z=z_{\tilde x \tilde x}-z$.
Multiplying by $\tilde{a}^0_{\tilde x}$ then integrating by parts gives $(\tilde{a}^0)^2 z=\tilde{a}^0_{\tilde x} z_{\tilde x}-\tilde{a}^0z$, which can be rearranged to $\tilde{a}^0_{\tilde x\tilde x} z = \tilde{a}^0_{\tilde x} z_{\tilde x}$.
Hence $z$ is a multiple of $\tilde{a}^0_{\tilde x}$. But this is an odd function, therefore $\ker(K_2-I)=\{0\}$ as required.
Therefore $K_2-I$ is continuously invertible from $C^0_\bddeven(\R)$ to $C^0_\bddeven(\R)$, and it is easy to then show that $K_2-I$ is continuously invertible from $C^2_\bddeven(\R)$ to $C^2_\bddeven(\R)$.

We can now write \eqref{eqn:ODEforz} as 
$\Phi(z,\eps)=0$, where 
\begin{align*}
    \Phi \colon C^2_\bddeven(\R) \times \R &\to C^2_\bddeven(\R),\qquad
    (z,\eps)\mapsto z-(I-K_2)^{-1}(K_1(r(z,\eps)))
\end{align*}
is a $C^{N-1}$ function. Taking a Fr\'echet derivative with respect to $z$ at $(0,0)$, we have that \begin{align*}
    D_z \Phi (0,0) &= I -(I-K_2)^{-1}(K_1(D_z r(0,0)))
    = I,
\end{align*}
an isomorphism. Therefore, we have by the implicit function theorem that $z^\eps$ exists for $\eps$ sufficiently small, and that the map $\eps \mapsto \tilde{a}^\eps$ is in $C^{N-1}(\R,C^2(\R))$. Note that in fact, we have $z^\eps$ exists for $\eps$ small and negative, but since our rescaling does not make sense for negative $\eps$, we ignore these solutions.
Applying the rescaling \eqref{eqn:RescaledCoords} to \eqref{eqn:bInTermsOfa:b}, we see that $\tilde b^\eps =\tilde{a}^\eps_{\tilde{x}}(1+\mathcal{O}(\eps))$,
where the $\mathcal{O}$ is with respect to the $C^0$ norm. However, considering the second component of \eqref{eqn:rescaledDirectODE}, we see that $\tilde{b}^\eps_{\tilde x}(\tilde x)$ has $C^N$ dependence on $(\tilde{a}^\eps(\tilde x), \tilde{a}^\eps_{\tilde{x}}(\tilde x)$, and $\eps$, so we see that in fact, 
\begin{align*}
    \tilde b^\eps &=\tilde{a}^\eps_{\tilde{x}}+\mathcal{O}(\eps)
    =\tilde{a}^0_{\tilde{x}}+\mathcal{O}(\eps),
\end{align*}
where the $\mathcal{O}$ is with respect to the $C^1$ norm.
\end{proof}
\bibliographystyle{plain} 
\bibliography{refs}

\end{document}